\documentclass[12pt,oneside]{amsart} 
\usepackage{amssymb,amscd}
\usepackage{euscript}

      \theoremstyle{plain}
      \newtheorem{theorem}{Theorem}[section]
      \newtheorem{lemma}[theorem]{Lemma}
      \newtheorem{corollary}[theorem]{Corollary}
      
      \newtheorem{remark}[theorem]{Remark}
      
      \newtheorem{definition}[theorem]{Definition}      
      \newtheorem{assumptions}[theorem]{Assumptions}     
\numberwithin{equation}{section}

      \makeatletter
      \def\@setcopyright{}
      \def\serieslogo@{}
      \makeatother

\def\M{X}
\def\c{\EuScript{C}} 
\def\R{\mathbb R}

\def\Z{\mathbb Z}

\def\N{\mathbb N}
\def\T{\mathbb T}

\def\G{\mathcal G}

\def\S{\mathcal {B}}
\def\po{{\mathcal{Q}}}
\def\r{{\mathcal{R}}}
\def\s{{\mathcal{S}}}
\def\ss{{\mathcal{SS}}}
\def\c{{\mathcal{C}}}
\def\n{{\mathcal{N}}}

\newcommand{\la}{\lambda}
\newcommand{\La}{\Lambda}

\def\Id{\text{Id}}
\def\e{\varepsilon}
\def\a{\alpha}

\def\Ci{C^\infty}

\def\Cr{C^{N,\alpha}}

\def\w{\mathcal{W}}

\def\f{\bar f}

\def\E{{\mathcal{E}}}

\def\V{{\mathcal{V}}}

\def\fe{{\mathcal {F}}}
\def\f{{\mathcal {F}}}
\def\fd{{F}}

\def\pe{{\mathcal {P}}}
\def\p{{\mathcal {P}}}
\def\pd{{P}}
\def\tp{ \tilde {\mathcal {P}}}
\def\tpd{ \tilde{P}}

\def\h{{\mathcal {H}}}
\def\hd{H}

\def\g{\mathcal G}
\def\G{\Gamma}

\def\QED{\hfill\hfill{\square}}

\begin{document}

%\date{\today}
\author{Boris Kalinin$^{\ast}$}

\address{Department of Mathematics, The Pennsylvania State University, University Park, PA 16802, USA.}
\email{kalinin@psu.edu}

\title[Non-stationary normal forms for contracting extensions]
{Non-stationary normal forms for contracting extensions} 
\dedicatory{Dedicated to the  memory of Anatole Katok}

\thanks{$^{\ast}$ Supported in part by Simons Foundation grant 426243}

%%%%%%%%%%%%%%%%%%%%%%%%%%%%%%%

\begin{abstract}
We present the theory of non-stationary normal forms for uniformly contracting smooth extensions
with sufficiently narrow Mather spectrum.  We give coherent proofs of existence, (non)uniqueness, 
and a description of the centralizer results. As a corollary, we obtain corresponding results
on normal forms along an invariant contracting foliation. The main improvements over the previous 
results in the narrow spectrum setting include explicit description of non-uniqueness and obtaining
results in any regularity above the precise critical level, which is especially useful for the centralizer. 
In addition to sub-resonance normal form, we also prove corresponding results for resonance normal 
form, which is new in the narrow spectrum setting.
\end{abstract}

\maketitle 

%%%%%%%%%%%%%%%%%%%%%%%%%%%%%%
%\newpage

\section{Introduction}

The theory of normal forms for smooth maps goes back to Poincare and Sternberg \cite{St}
and plays an important role in dynamics. It has been extensively studied in the classical 
setting of normal forms at fixed points and invariant manifolds \cite{BK}. The theory of 
non-stationary normal forms was developed more recently in the context of extensions
and invariant foliations. The primary motivations and applications were various rigidity results
for systems and actions exhibiting some form of hyperbolicity.

In the setting of an invariant contracting foliation $\w$ for a diffeomorphism $f$ of a compact manifold 
$\M$, the goal is to obtain a family of diffeomorphisms $\h_x: \w_x \to T_x\w$ such that the maps 
\begin{equation} \label{form}
 \p_x =\h_{fx} \circ f \circ \h_x ^{-1}: \;T_x \w \to T_{fx}\w
\end{equation}
are as simple  as possible. The maps $\p_x$, called the normal form of $f$ on $\w$, will be 
polynomial in our setting. The case of linear $\p_x$ is called non-stationary linearization. 
Some of the theory can be developed in a more general context of a smooth extension $\f$ 
of $f$ to a vector bundle $\E$ over $X$. The foliation setting produces such a smooth extension 
on the tangent bundle to the foliation $\E=T\w$ as follows.  We take $\f _x = f|_{W_x} : \E_x \to \E_{fx}$ 
after local identification of the leaf $\w_x$ with its tangent 
space $\E_x=T_x\w$.  In this setting, the map $\h_x$ is a coordinate change on $\E_x$.
%and $\p_x=\h_{fx} \circ \f_x \circ \h_x ^{-1}: \E_x  \to \E_{fx}.$ 

\vskip.2cm

The non-stationary linearization for one-dimensional extensions was obtained by Katok and 
Lewis \cite{KL} and applied to the study of rigidity for $SL(n,\Z)$ actions on $\T^n$. 
For higher-dimensional foliations under the assumption of constant 1/2 pinched bounds
for contraction rates, non-stationary linearization follows from results of Guysinsky and Katok
\cite{GK} or from results of Feres in \cite{F1}. Under a weaker assumption of pointwise 1/2 pinching, 
it was obtained by Sadovskaya \cite{S} and some further properties were established by Kalinin and  
Sadovskaya \cite{KS}. This is a less technical but very important case of non-stationary normal forms 
and we give a brief summary of these results in Section~\ref{NSLinear}. They were used extensively
 in the study of rigidity of Anosov systems  and higher rank actions, see e.g. \cite{S,KS03,KS,Fa,FFH,GKS10,But}.

\vskip.2cm

In higher-dimensional setting without 1/2 pinching, there may be no smooth non-stationary 
linearization, and so a polynomial normal form is sought. Under the narrow band spectrum 
assumption, such forms were developed by Guysinsky and Katok \cite{GK,G} and used by 
Katok and Spatzier to obtain local rigidity of algebraic higher rank Anosov abelian actions \cite{KSp97}. 
A geometric point of view on normal forms was developed by Feres in \cite{F2}.
The narrow band assumption %is satisfied, for example, for perturbations  of algebraic systems and 
ensures that the polynomial maps involved belong to a finite dimensional Lie group of so called sub-resonance
generated polynomials. In \cite{KS15} Kalinin and Sadovskaya obtained stronger results, constructing 
$\h_x$ which depend smoothly on $x$ along the leaves and proving that they define an atlas with 
transition maps in a similar finite dimensional Lie group. 

\vskip.2cm

A parallel theory of  non-stationary normal forms was also developed for non-uniform contractions. 
Basic results were formulated by Kalinin and Katok \cite{KKt00} along with a program of applications
 to measure rigidity for non-uniformly hyperbolic systems and actions. The existence of $\h_x$ for a 
 general contracting $C^\infty$ extension was proved by Li and Lu \cite{LL} in the setting of random 
 dynamical systems. Some results, such as existence of Taylor polynomial or formal series for $\h_x$, 
 can be obtained for extensions more general than contractions, see \cite{AK,A,LL}.
Non-stationary linearization of a $C^{1+\a}$ diffeomorphism along a non-uniformly contracting foliation 
was obtained by Kalinin and Katok \cite{KKt} for one-dimensional leaves and by Katok and Rodriguez 
Hertz \cite{KtR} under pinching assumption on Lyapunov exponents. The former result was used in the 
study of measure rigidity by Kalinin, Katok, and Rodriguez Hertz  \cite{KKt,KKtR}. 
Following \cite{KS15}, the advanced theory of normal forms on non-uniformly contracting foliations,
including the consistency of normal form coordinates along the leaves, were  obtained
independently by Kalinin and Sadovskaya \cite{KS16} and, in differential geometric formulations, 
by Melnick \cite{M}. In addition to measurable properties of non-uniformly hyperbolic systems, this 
theory is useful in global, rather than local, smooth rigidity of uniformly hyperbolic systems, where
the  spectrum may not be narrow \cite{FKSp10}. The main disadvantage of this setting is that 
dependence of $\h_x$ on $x$ is only measurable.

\vskip.2cm

Our main goal is to present the theory of non-stationary normal forms, both in extension and foliation settings, 
assuming that the spectrum is sufficiently narrow. In particular, the results apply to perturbations of algebraic 
and point spectrum systems. We follow the approach developed in  \cite{KS15,KS16} and give a coherent treatment of existence, (non)uniqueness, and centralizer results. The main improvements over the previous results in the narrow band spectrum
setting are the following. Our construction allows us to describe the exact extent of non-uniqueness 
in $\h_x$ and $\p_x$ and hence gives the description of the centralizer. It works in any regularity of 
$\f$ above the precise critical level in H\"older classes. This is especially important for the centralizer
results as they yield an automatic bootstrap of regularity for a commuting system, from critical one
to that of $\h_x$. These improvements proved useful in rigidity results \cite{DWX,GKS19}.
In addition to sub-resonance normal form, we also prove existence, (non)uniqueness, and 
centralizer results for resonance normal form, which has not been done in the  narrow band setting.

%%%%%%%%%%%%%%%%%%%%%%%%%%%%%%
%%%%%%%%%%%%%%%%%%%%%%%%%%%%%%

\section{Non-stationary linearization} \label{NSLinear}
%%%%%%%%%%%%% Uniform %%%%%%%%%
In this section we summarize the results on non-stationary linearization, that is existence of $\h_x$ 
so that $\p_x$ in \eqref{form} are linear. This important particular case of non-stationary normal
form theory is easier to formulate and provides a point of comparison for the more technical general case.
Also, it is the only result obtained under pointwise assumption on contraction rates. We formulate 
it in foliation setting, where it is most interesting. We state the results for $\Ci$ case, as established 
in \cite[Proposition 4.1, Lemma 4.1]{S} and \cite[Proposition 4.6]{KS}, and then make remarks on lower regularity.

\begin{theorem} [Non-stationary linearization]  \cite{S,KS} \label{Linear} $\;$ \\
Let $f$ be a diffeomorphism of a compact manifold $\M$ and let $\w$ be an $f$-invariant continuous
 foliation of $\M$ with uniformly $\Ci$ leaves. Suppose that $\| Df|_{T\w}\|<1$, and there exist
$C>0$ and $\gamma<1$ such that 
\begin{equation}\label{12pinch}
 \|\,(Df^n|_{T_x\w})^{-1}\,\| \cdot \|\,Df^n|_{T_x\w}\,\|^2   \leq C\gamma^n  \quad \text{ for all }\; x\in \M, \, n\in\mathbb N.
  \end{equation}
  Then for every $x\in \M$ there exists a $C^\infty$ diffeomorphism 
$\h_x: \w_x \to T_x\w$ such that
\begin{itemize}
 \item[(i)] $\h_{fx}\circ f \circ \h_x^{-1}=Df|_{T_x\w} ,$ 
 \vskip.05cm
 \item[(ii)] $\h_x(x)=0$  and $D_x\h_x$ is the identity map, 
 \vskip.05cm  
 \item[(iii)] $\h_x$ depends continuously on $x\in \M$ in $C^\infty$ topology.
  \vskip.05cm
 \item[(iv)]  Such a family $\h_x$ is unique and 
depends smoothly on $x$ along the leaves of $\w$.
 \vskip.05cm
  \item[(v)] The map $\h_y \circ \h_x^{-1}: T_x\w\to T_y\w$
is affine for any $x\in \M$ and $y\in \w_x$. Hence the non-stationary linearization $\h$ defines affine structures on the leaves of $\w$.
\end{itemize}
\end{theorem}

We note that non-stationary linearization is the only case when the $\h_x$, and hence $\p_x$, are unique 
under assumptions (i) and (ii) and sufficient regularity. This uniqueness immediately implies the description of
centralizer as in part (3) of Theorem \ref{NFfol}. Together with (v), it also easily gives
smooth dependence on $x$ along the leaves of $\w$.

Under stronger 1/2 pinching assumption on rates in place of  \eqref{12pinch},  finite regularity version 
follows from our general results, see Corollary \ref{cor1/2pinch}. However, finite regularity results can 
also be obtained using \eqref{12pinch}:

\begin{remark} \label{LinearR} 
The proof of existence of $\h$ in \cite[Proposition 4.1]{S} is for $C^N$ with $N$ sufficiently large
and can be seen to work for $N=2$. 
%Under assumption \eqref{12pinch} with $1+\beta$one can expect existence if $f$ is $C^{1+\beta}$.

Parts {\em (iv)} and {\em (v)} hold under the assumption that $f$ and $\h_x$ are $C^2$. 
This is clear from the proof of \cite[Lemma 4.1]{S} and \cite[Proposition 4.6]{KS}. More precisely, 
it suffices to assume that  $Df|_{T_x\w}$ and $D\h_x$ are Lipschitz along the leaves of $\w$ with
uniform constant for all local leaves. More generally, uniqueness holds if they are $\a$-H\"older
under stronger pinching assumption with the term $\|\,Df^n|_{T_x\w}\,\|^2$ in \eqref{12pinch} 
replaced by $\|\,Df^n|_{T\w}\,\|^{1+\a}$ (cf. Corollary \ref{cor1/2pinch}). In the particular case 
when $\w$ is a one-dimensional foliation, the uniqueness holds if $\h$ is $C^1$ \cite{KL}.
\end{remark}

We note that equation \eqref{12pinch} with $1+\beta$ in place of $2$ is precisely the 
{\em fiber bunching} assumption  on the linear cocycle $Df|_{T_x\w}$ relative to the contraction 
along $\w$, which in particular ensures existence of cocycle holonomies along $\w$. 
In fact, the holonomies are given by the derivatives of the transition maps in  (v).

\section{Preliminaries and notations}\label{Prelim} 

\subsection{Smooth extensions} \label{Smooth ext}
Let $\E$ be a continuous vector bundle over a compact metric space $\M$, let  $\V$ be
 a neighborhood of the zero section in $\E$, and let $f$ be a homeomorphism of $\M$. 
We consider an extension $\f : \V \to \E$ that projects to $f$ and preserves the zero section.
We assume that the corresponding fiber maps $\f_x: \V_x \to \E_{f(x)}$ are $C^r$ diffeomorphisms.

 \vskip.1cm

% $\Cr (B_{x,\sigma })=\Cr (B_{x,\sigma}, \E_x)$ 
If $r=N\in \N$, we will assume that $\f_x$ depend continuously on $x$ in $C^N$ topology. 
To obtain sharper results, we will also use {\em H\"older condition at the zero section.} 
We assume that the fibers $\E_x$ are equipped with a continuous family of Riemannian norms.
We denote by $B_{x,\sigma }$ the closed ball of radius $\sigma >0$ centered at 
$0 \in \E_x$. For $N \in \N$ and $0\le \a \le1$ we denote by $\Cr (B_{x,\sigma })=\Cr (B_{x,\sigma}, \E_{fx})$ 
the space of functions $R: B_{x,\sigma}\to \E_{fx}$ with continuous derivatives up to order 
$N$ on $B_{x,\sigma}$ and, if $\a>0$, with $N^{th}$ derivative satisfying $\a$-H\"older condition at $0$:
\begin{equation}\label{Canorm}
\| D^{(N)} R\|_\a =\sup \,\{ \,\| D^{(N)}_t R - D^{(N)}_0 R\|\cdot \|t\|^{-\a} : \; 0\ne t \in B_{x,\sigma}\} < \infty.
  \end{equation}
 We call $\| D^{(N)} R\|_\a$ the $\a$-H\"older constant of $D^{(N)} R$ at $0$.
We equip the space $\Cr (B_{x,\sigma })$ with the norm
\begin{equation}\label{Crnorm}
  \|R\|_{\Cr (B_{x,\sigma })}=\,
  \max\, \{\, \|  R\|_0, \;\| D^{(1)} R\|_0,\; ..., \;\| D^{(N)} R\|_0, \;\| D^{(N)} R\|_\a \, \},   \end{equation}
where $\| D^{(k)} R\|_0=\sup\, \{ \| D^{(k)}_t R\|: \; t \in B_{x,\sigma}\}$ and the last term is omitted if $\a=0$.

  \vskip.1cm
  
\begin{definition} \label{Cr ext} 
We say that $\f$ is a $\Cr$ extension of $f$, $N \in \N$ and $0\le\a \le1$,
if for some $\sigma>0$ the fiber maps $\f_x:  B_{x,\sigma} \to \E_{f(x)}$ are $\Cr$ diffeomorphisms which depend continuously on $x$ in $C^N$ topology and the norms $\|\f _x\|_{\Cr (B_{x,\sigma })}$ are uniformly bounded.

Similarly, we say that $\h =\{\h_x\}_{x\in X}$, where $\h_x:  B_{x,\sigma} \to \E_{f(x)}$, is a $\Cr$ coordinate change if it is a $\Cr$ extension of $f=\Id$ which preserves the zero section.
\end{definition}

%We consider an extension $\f$ satisfying the Assumptions \ref{ass}. If $N\ge 2$ we allow $\a=0$, in which case we understand $\Cr$ as $C^N$.

\subsection{Mather spectrum of the derivative} \label{Mather}

For a smooth extension $\f$ we will denote by $F$ its derivative of  at the zero section, that is 
$F : \E \to \E$ is a continuous linear extension of $f$ whose fiber maps are linear isomorphisms
$F_x =D_0 \f_x  : \E_x \to \E_{fx}$. Such a linear extension $F$ induces a bounded linear operator 
$F^*$ on the space of continuous sections of $\E$ by $F^*v(x)=F(v(f^{-1}x))$. The spectrum $Sp \, F^*$  of 
complexification of $F^*$ is called {\em Mather spectrum} of $\fd$. Under a mild assumption that 
non-periodic points of $f$ are dense in $\M$, the Mather spectrum consists of finitely many closed 
annuli centered at $0$, see e.g. \cite{P}, and its {\em characteristic set}
$\La (F) = \{ \la \in \R : \exp \la \in Sp \, F^*\}$ consists of finitely many closed intervals. 
We will assume that $F$ is a contraction and that the Mather spectrum of $F$ is sufficiently narrow. 
\begin{definition} \label{chi ext} 
Let $\e>0$ and $\chi = (\chi_1, \dots, \chi_\ell)$, where $\chi_1<\dots<\chi_\ell<0$.
We say that a linear extension $F$ has {\em $(\chi,\e)$-spectrum} if 
\begin{equation} \label{Lachi}
\La (F) =  \{ \la \in \R : \exp \la \in Sp \, F^*\} \subset \bigcup _{i=1}^\ell (\chi_i-\e,\chi_i+\e)
\end{equation}
\end{definition}
\noindent If $F$ has $(\chi,\e)$-spectrum with disjoint intervals then the bundle $\E$ splits into direct sum 
\begin{equation} \label{splitting}
\E=\E^{1} \oplus \dots \oplus \E^{\ell}
\end{equation} 
of continuous $F$-invariant sub-bundles so that $\La (F|_{\E^i})$ is contained in $(\chi_i-\e,\chi_i+\e)$.  
%We denote by $t=(t_1,\dots,t_l)$ the corresponding splitting of a vector  $t \in \Ex$ into the components. 
%? Since the splitting is invariant, the linear map $\fd (x): \E_{x} \to \E_{fx}$  has ``block-diagonal" form 
% $$\fd_x= \fd^1_x \oplus \dots  \oplus \fd^l_x, \quad\text{where} \quad \fd^i (x): \E^i_{x} \to \E^i_{fx}.$$
This can be expressed using a convenient metric \cite{GK}:  for each 
$i=1,..., \ell$ there exists a continuous family of Riemannian norms 
$\|.\|_{x}$ on $\E_x$ such that the splitting \eqref{splitting} is orthogonal and
 \begin{equation} \label{estAEi}
  e^{\chi _i -\e} \| t \|_{x} \le  \| \fd _x (t)\|_{fx} \le e^{\chi _i +\e} \| t\|_{x}
  \quad\text{for every } t \in \E^i_x.
\end{equation}
We will equip $\E$ with such a norm and will suppress the dependence on $x$. We can also summarize \eqref{estAEi} using operator norms
\begin{equation}  \label{estAnorm}
  \| F|_{\E^i_x} \| \le e^{\chi _i +\e}, \quad   \| (F|_{\E^i_x})^{-1} \| \le e^{-\chi _i +\e}, \quad  
\| F_x \| \le e^{\chi _\ell +\e}, \quad   \| (F_x)^{-1} \| \le e^{-\chi _1 +\e}.
\end{equation}

%We will not consider $\ell=1$ as stronger results exist for this case.  ... ?

 \subsection{ Sub-resonance and resonance polynomials} \label{polynomials}

We say that a map between vector spaces is {\em polynomial}\, if each component  is given by a polynomial 
in some, and hence every, basis.
We will consider a polynomial map $P: \E_x \to \E_y$ with $P(0_x)=0_y$ and 
split it into components $(P_1(t),\dots,P_{\ell}(t))$, where $P_i: \E_x \to \E_y^i$. 
Each $P_i$ can be written uniquely as a linear combination of  polynomials
of specific homogeneous types.  We say that $Q: \E_x \to \E_y^i $ has {\em homogeneous type} 
$s= (s_1, \dots , s_\ell)$, where $s_1,\dots,s_{\ell}\,$ are non-negative integers, 
if for any real numbers  $a_1, \dots , a_\ell$ and vectors
$t_j\in \E_x^j$, $j=1,\dots, \ell,$ we have 
\begin{equation}\label{stype}
Q(a_1 t_1+ \dots + a_\ell t_\ell)= a_1^{s_1} \cdots  a_\ell^{s_\ell} \cdot Q( t_1+ \dots + t_\ell).
\end{equation}

\begin{definition} \label{SRdef}
We say that a homogeneous type $s= (s_1, \dots , s_\ell)$ for $P_i: \E_x \to \E_y^i$  is  
\begin{equation}\label{sub-resonance}
\text{{\bf sub-resonance} if} \quad \chi_i \le \sum_{j=1}^\ell s_j \chi_j, \quad \text{and \;  {\bf resonance} if} 
\quad \chi_i = \sum_{j=1}^\ell s_j \chi_j.
\end{equation}
We say that a polynomial map $P: \E_x \to \E_y$ is {\em sub-resonance (resp. resonance)}
if each component $P_i$ has only terms of  sub-resonance (resp. resonance) homogeneous types.
We denote by $\s_{x,y}$ (resp. $\r_{x,y}$) the set of all sub-resonance (resp. resonance) polynomials 
$P: \E_x \to \E_y$ with $P(0)=0$ and invertible derivative at $0$.
\end{definition}

Clearly, for any sub-resonance relation we have $s_j=0$ for $j<i$ and $\sum s_j \le \chi_1/ \chi_\ell$.
It follows that sub-resonance polynomials have degree at most 
\begin{equation}\label{degree}
d=d(\chi)= \lfloor \chi_1/\chi_\ell \rfloor.
\end{equation}
We will denote  $\s_{x,x}$ by $\s_x$, which is a finite-dimensional Lie group
 group with respect to the composition \cite{GK}. All groups $\s_x$ are isomorphic, 
 moreover, any map $P\in \s_{x,y}$ induces an isomorphism between $\s_x$ and $\s_y$
by conjugation. Any invertible linear map $A:\E_y \to \E_x$ which respects the splitting 
induces an isomorphism between the groups $\s_x$ and $\s_y$. 
 Similar statements hold for the resonance groups $\r_x=\r_{x,x}$.

%Of course, if the bundle $\E$ and the sub-bundles $\E^i$ are trivial, then all $G_x$ and $G_{x,y}$ can be identified with a single group $G^{\xf,\xs}(E)$.

% {\bf Note:} Sub-resonance in ... is defined in terms of $ \chi_i\pm\e'$ in place of $\chi_i$. We will choose $\e'$ sufficiently small so that the two notions coincide. In particular, in our case both $\s$ and $\r$ are groups.  

%(``block diagonal) (``block triangular)
Note that a linear map is resonance, resp. sub-resonance, if and only if it preserves the splitting 
 \eqref{splitting}, resp. the associated flag of fast sub-bundles:
\begin{equation}\label{fastflag}
\E_x^1=\V_x^1 \subset \V_x^2 \subset ... \subset \V_x^l =\E_x, \quad \text{where }\; \V^i_x= \E^1_x \oplus \dots \oplus \E^i_x
\end{equation} 
While the notion of resonance polynomials depends on the splitting, the notion of sub-resonance 
polynomials depends only on the flag \eqref{fastflag} and sub-resonance polynomials preserve the flag, 
see \cite[Proposition 3.2]{KS16}.

 \subsection{Narrow spectrum} \label{narrow}

Now for a given $\chi = (\chi_1, \dots, \chi_\ell)$, where $\chi_1<\dots<\chi_\ell<0$,
we will define $\e_0=\e_0 (\chi)>0$ which ensures that the spectrum is sufficiently narrow.
Informally, we choose it so that if $F$ has $(\chi,\e)$-spectrum for some $\e<\e_0 (\chi)$ 
then all Mather spectrum sub-resonances and resonances of $F$ come from the point spectrum 
$\chi$ and also any non-resonance homogeneous type is contracted by forward or backward 
iterates of $F$. This condition is stronger than the narrow band spectrum in \cite{GK}.

We define $\tilde \la<0$ as the largest value of $-\chi_i + \sum_{j=1}^\ell s_j \chi_j$ over all  
$i\in \{1, \dots, \ell \}$ and non-negative integers  
$s_1,\dots,s_{\ell}\,$ such that this value is negative, that is, they do not satisfy any 
sub-resonance relation \eqref{sub-resonance}:
\begin{equation}\label{lambda}
\tilde \la =\max \, \{ -\chi_i + \sum s_j \chi_j <0\,\} \; \text{ and let} \;  \la = \max \, \{ \tilde \la,  -\chi_1+ (d+1) \chi_\ell \}<0.
\end {equation} 
The maximum exists since there are at most finitely many 
values of  $-\chi_i + \sum s_j \chi_j $ greater than any given number.

Similarly, we define $\mu<0$ as the largest value of $\chi_i - \sum_{j=1}^\ell s_j \chi_j$ 
over all  $i\in \{1, \dots, \ell \}$ and non-negative integers  
$s_1,\dots,s_{\ell}\,$ such that this value is negative, that is, they satisfy some sub-resonance 
relation which is not a resonance one \eqref{sub-resonance} (we will refer to such 
homogeneous types as {\em strict sub-resonance}):
\begin{equation}\label{mu}
\mu =\max \, \{ \chi_i - \sum s_j \chi_j <0\,\}.
\end {equation} 
The maximum exists since there are at most  finitely many sub-resonance relations.
Finally, we define
\begin{equation}\label{epsilon}
\e_0=\e_0 (\chi)=\min\,\{\, -\chi_\ell, -\lambda/(d+2) , -\mu/(d+1) \,\}>0.
\end{equation}

\section{Statements of results}\label{Snormalforms} 

First we summarize our basic notations and assumptions.

\begin{assumptions} \label{ass} In this section, \\
$f:X\to X$ is a homeomorphism of a compact metric space $X$,\\
$\E$ is a continuous vector bundle over $X$, equipped with a continuous Riemannian metric,\\
$\V$ is a neighborhood of zero section in $\E$ and $B_{x,\sigma }  \subset  \V_x$ for some $\sigma >0$ and all $x\in X$, \\
$\f:\V  \to \E$ is a $\Cr$ extension of $f$ (see Def. \ref{Cr ext}) that preserves the zero 
 section,  \\
%$\f_x$ is in $\Cr (B_{x,\sigma },\E_{fx})$, see \eqref{Crnorm}, and depends continuously on $x$ in $\Cr$ topology,\\
$\f$ is contracts in the sense $\| \f_x(t)\| \le \xi \|t\| $ for some $\xi <1$,  and all $x\in X$ and  $t\in B_{x,\sigma }$, \\
 $F:\E\to \E$ is the derivative of  $\fe$ at the zero section, 
 $F_x=D_0\fe_x :\E_x \to \E_{fx}$, \\
$F$ has $(\chi,\e)$-spectrum (see Def. \ref{chi ext}) for some $\chi = (\chi_1, \dots, \chi_\ell)$, with $\chi_1<\dots<\chi_\ell<0$, 
and some $\e<\e_0=\e_0 (\chi)$ given by \eqref{epsilon}.
%that is the Mather spectrum of $F$ satisfies  \eqref{Lachi}. 
 \end{assumptions}

\begin{remark} \label{assSp} As we outlined in Section \ref{Mather}, if $F$ has $(\chi,\e)$-spectrum 
then there is a continuous Riemannian metric on $\E$ and  a continuous orthogonal 
$F$-invariant splitting $\E=\E^{1} \oplus \dots \oplus \E^{l}$ such that for each $i=1,..., \ell$ 
and all $x\in X$ we have
 \begin{equation} \label{estAEi'}
  e^{\chi _i -\e} \| t \| \le  \| \fd _x (t)\| \le e^{\chi _i +\e} \| t\|
  \quad\text{for every } t \in \E^i_x,
\end{equation}
which also yields \eqref{estAnorm}. This is the property that we use in the proof of Theorem \ref{NFext}. 
\end{remark}

We recall that $\Cr$ is the space of $C^N$ functions with $N^{th}$ derivative satisfying 
$\a$-H\"older condition at $0$ \eqref{Canorm}.
We will require that the smoothness $N+\a$ is higher than the ``critical regularity" 
$\chi_1 / \chi_\ell \ge 1$. If $N\ge 2$ we allow $\a=0$. 

%and set \begin{equation}\label{nu} \nu=\chi_1-(N+\alpha)\chi_\ell\,>0. \end{equation}
%If $N\ge 2$ we allow $\a=0$, in which case we understand $\Cr$ as $C^N$.

\begin{theorem}[Normal forms for contracting extensions]\label{NFext} $\;$ \\
Let $\f$ be an extension of $f$ satisfying Assumptions \ref{ass}. Suppose that $N\in\N$, $0\le \a \le1$,
\begin{equation}\label{crit}
 \nu=\chi_1-(N+\alpha)\chi_\ell\,>0  \quad \text {and }\quad  \e< \nu/( N+\a+1).
\end{equation}
Then 
\noindent {\bf (1)}  There exists a $\Cr$ coordinate change 
$\h=\{ \h_x\}_{x\in X}$ (see Def. \ref{Cr ext}) with diffeomorphisms   
 $\h_x : B_{x,\sigma} \to \E_x$  satisfying $\h_x(0)=0$ and $D_0 \h_x =\Id \,$
 which conjugates $\fe$ to a continuous polynomial extension $\pe$ of sub-resonance type
 (see Def. \ref{SRdef}): 
 \begin{equation}\label{nfs}
\h_{fx} \circ \f_x =\p_x \circ \h_x, \; \text{ where }  \;  \p_x\in \s_{x,fx}
 \; \text{ for all } x\in X. 
\end{equation}
\vskip.1cm

\noindent {\bf (1')}  There exists a $\Cr$ coordinate change 
$\h'=\{ \h_x'\}_{x\in X}$ with diffeomorphisms   
 $\h_x' : B_{x,\sigma} \to \E_x$  satisfying $\h_x'(0)=0$ and $D_0 \h_x' =\Id \,$
 which conjugates $\fe$ to a continuous polynomial extension $\pe'$ of resonance type: 
  \begin{equation}\label{nfr}
\h_{fx}' \circ \f_x =\p_x' \circ \h_x', \; \text{ where }  \;  \p_x\in \r_{x,fx}
 \; \text{ for all } x\in X.
\end{equation}
 
\vskip.1cm

\noindent {\bf (2)} Suppose  $\tilde \h=\{ \tilde \h_{x}\}_{x\in X}$  is another $\Cr$ coordinate 
change as in (1) conjugating $\f$ to a sub-resonance polynomial extension  $ \tilde \p$. 
Then there exists a continuous family $\{ G_x\}_{x\in X}$ with $G_x \in \s_x$ such that 
$\h_x=G_x \circ \tilde \h_{x}$. Moreover, if $D^{(n)}_0\tilde \h_x=D^{(n)}_0\h_x$ for all 
$n=2,...,d= \lfloor \chi_1/\chi_\ell \rfloor$, then $\h_x= \tilde \h_{x}$ for all $x \in X$. 

\vskip.1cm

\noindent {\bf (2')} Suppose  $\tilde \h'=\{ \tilde \h_{x}'\}_{x\in X}$ is another $\Cr$ coordinate 
change as in (1') conjugating $\f$ to a resonance polynomial extension  $ \tilde \p'$. 
Then there exists a continuous family 
$\{ G_x'\}_{x\in X}$ with $G_x' \in \r_x$ such that 
$\h_x'=G_x' \circ \tilde \h_{x}'$. Moreover, if 
$D^{(n)}_0\tilde \h_x'=D^{(n)}_0\h_x'$ for all $n=2,...,d= \lfloor \chi_1/\chi_\ell \rfloor$, then
$\h_x'= \tilde \h_{x}'$ for all $x \in X$. 
%In particular, $\{ \h_x\}_{x\in X}$ is unique if $\,d=1$.

\vskip.1cm

\noindent  {\bf (3)} 
Let $g:X\to X$ be a homeomorphism commuting with $f$ and $\g:\V  \to \E$ be a $C^{N',\a'}$ 
extension of $g$ which preserves the zero section and commutes with $\f$.
Suppose that $N'\in \N$ and $0\le \a' \le1$ satisfy $N' \le N$, $N'+\a' \le N+\a$, and
\begin{equation}\label{crit'}
  \nu'=\chi_1-(N'+\alpha')\chi_\ell\,>0  \quad \text {and }\quad  \e< \nu'/( N'+\a'+1).
\end{equation}
Then the coordinate changes $\h$ and $\h'$ conjugate $\g$ to continuous
sub-resonance  and resonance polynomial extension respectively, that is 
 \begin{equation}\label{cent}
\h_{gx} \circ \g_x \circ \h_x^{-1} \in \s_{x,fx} \; \text{ and }  \;  \h_{gx}' \circ \g_x \circ (\h_x')^{-1} \in \r_{x,fx}.
 \; \text{ for all } x\in X.
\end{equation}
In particular, $\g$ is a $\Cr$ extension

\end{theorem}

\begin{corollary} \label{Cinf}
Suppose that $\f$ in the theorem is a $\Ci$ extension. Then the coordinate 
changes $\h$ in part (1) and $\h'$ in part (1')  are also $\Ci $.
\end{corollary}

\begin{remark} [Global version] \label{NFglob}
Suppose that $\f :\E\to \E$ is a globally defined extension which satisfies the assumptions of 
Theorem \ref{NFext} and either contracts fibers or, more generally, satisfies the property that 
for any compact set $K \subset \V$ and any 
neighborhood $U$ of the zero section we have $\f^n(K) \subset U$ for all sufficiently large $n$.
Then the coordinate changes $\h$ and $\h'$ can be uniquely extended   ``by invariance"
$\h_x  = (P_x^n)^{-1} \circ  \h_{f ^n (x)} \circ \f_x^n$ to the family of global $\Cr$  coordinate changes 
with diffeomorphisms $\h_x,\h'_x :  \E_{x} \to \E_{x}$ satisfying \eqref{nfs} and \eqref{nfr} respectively.
Moreover, if the extension $\g$ as in (3) is also globally defined, then it satisfies \eqref{cent} globally.
 \end{remark}

\subsection{Normal forms for contracting foliations}

Now we apply the results above in the context of diffeomorphisms with invariant contracting foliation. 

Let $f$ be a $C^r$ diffeomorphism of a compact manifold  $\M$. We will consider $r>1$,
and for $r \notin \N$ we will understand $C^r$ in the usual sense that the derivative of order 
$N=\lfloor r \rfloor$ is H\"older with exponent $\a=r-\lfloor r \rfloor$. We will consider an $f$-invariant 
continuous foliation $\w$ of $X$ with {\em uniformly $C^r$ leaves}, by which we mean that for some $R>0$
the balls $B^\w(x,R)$ of radius $R$ in the intrinsic Riemannian metric of the leaf can be given by $C^r$ embeddings which depend continuously on $x$ in $C^N$ topology and, if $r \notin \N$,
have $\a$-H\"older derivative of order $N$ with uniformly bounded H\"older constant.
Similarly, for such a foliation we will say that a function $g$ is {\em uniformly $C^r$ along $\w$} 
if its restrictions to $B^\w(x,R)$ depend continuously on $x$ in $C^N$ topology and
have $\a$-H\"older derivative of order $N$ with uniformly bounded H\"older constant.
We also allow $r=\infty$, in which case uniformly $\Ci$ means uniformly $C^N$ for each $N$.

%The important new statements in this setting describing dependence along the leaves, parts (2) and (3) in the next theorem, were established in \cite{KS16}.

\begin{theorem}[Normal forms for contracting foliations]\label{NFfol} 
Let $f$ be a $C^r$, $r\in (1,\infty ]$, diffeomorphism of a smooth compact manifold $\M$,
and let $\w$ be an $f$-invariant topological foliation of $\M$ with uniformly $C^r$ leaves. 
Suppose that %$\w$ is contracted by $f$, i.e. $\| Df |_{T\w}\|<1$ for some metric, and that 
the linear extension $F=Df |_{T\w}$ has $(\chi,\e)$-spectrum  (or alternatively F satisfies the condition in 
Remark \ref{assSp}) for some $\chi = (\chi_1, \dots, \chi_\ell)$, 
where $\chi_1<\dots<\chi_\ell<0$, and some $\e<\e_0=\e_0 (\chi)$ given by \eqref{epsilon}.
%that is the Mather spectrum of $F$ satisfies \eqref{Lachi}.  
Suppose also that $r> \chi_1/ \chi_\ell$ and
\begin{equation}\label{crit r}
 \e< \nu/(r+1),   \quad \text {where }\quad \nu=\chi_1-r \chi_\ell\,>0.
\end{equation}

\noindent Then {\bf (1)} There exists a family $\{ \h_x \} _{x\in \M}$ of $C^r$ diffeomorphisms  
$\,\h_x: \w_x \to T_x\w$ satisfying $\h_x(0)=0$ and $D_0 \h_x =\Id \,$ such that  for each $x \in \M$,
 $$
 \p_x =\h_{f(x)} \circ f \circ \h_x ^{-1}:T_{x}\w \to T_{f{(x)}}\w \text{ is in  }\s_{x,fx}. 
 $$
The maps , $\,\h_x |_{B^\w(x,R)}$ depend continuously on $x \in \M$ in $C^N$ topology with
$N=\lfloor r \rfloor$ and, if $\a=r-N>0$, they have $\a$-H\"older derivative of order $N$ with 
uniformly bounded H\"older constant.
 \vskip.2cm

\noindent  {\bf (1')} There exists a family $\{ \h'_x \} _{x\in \M}$ of  diffeomorphisms  
as in (1)  so that  for all $x \in \M$,
 $$
 \p_x =\h_{f(x)}' \circ f \circ (\h_x') ^{-1}:T_{x}\w \to T_{f{(x)}}\w \text{ is in  }\r_{x,fx}. 
 $$

\noindent {\bf (2)} Suppose  $\tilde \h=\{ \tilde \h_{x}\}_{x\in X}$ is another family of 
diffeomorphisms as in (1) conjugating $f|_\w$ to sub-resonance polynomials $ \tilde \p_x \in \s_{x,fx}$. 
Then there exists a continuous family 
$\{ G_x\}_{x\in X}$ with $G_x \in \s_x$ such that 
$\h_x=G_x \circ \tilde \h_{x}$. Moreover, if 
$D^{(n)}_0\tilde \h_x=D^{(n)}_0\h_x$ for all $n=2,...,d= \lfloor \chi_1/\chi_\ell \rfloor$, then
$\h_x= \tilde \h_{x}$ for all $x \in X$. 
%In particular, $\{ \h_x\}_{x\in X}$ is unique if $\,d=1$.

\vskip.2cm

\noindent {\bf (2')} Suppose  $\tilde \h'=\{ \tilde \h_{x}'\}_{x\in X}$  is another family of 
diffeomorphisms as in (1') conjugating $f|_\w$ to resonance polynomials $ \tilde \p_x \in \r_{x,fx}$
Then there exists a continuous family 
$\{ G_x'\}_{x\in X}$ with $G_x' \in \r_x$ such that 
$\h_x'=G_x' \circ \tilde \h_{x}'$. Moreover, if 
$D^{(n)}_0\tilde \h_x'=D^{(n)}_0\h_x'$ for all $n=2,...,d= \lfloor \chi_1/\chi_\ell \rfloor$, then
$\h_x'= \tilde \h_{x}'$ for all $x \in X$. 
%In particular, $\{ \h_x\}_{x\in X}$ is unique if $\,d=1$.

\vskip.2cm
\noindent {\bf (3)} Let $g$ be a homeomorphism of $\M$ which commutes with $f$,  preserves $\w$, 
 and is uniformly $C^{r'}$  along the leaves of $\w$. Suppose that $1<r'\le r$ satisfies
\begin{equation}\label{crit r'}
\nu'=\chi_1-r' \chi_\ell\,>0   \quad \text {and }\quad \e< \nu '/(r'+1).
\end{equation}
Then $Q_x=\h_{g(x)} \circ g \circ \h_x ^{-1} \in \s_{x,gx}$ and 
 $Q_x'=\h_{g(x)}' \circ g \circ \ (\h_x') ^{-1}\in \r_{x,gx}$ for all $x \in \M$.  
 In particular, $\g_x$ is in is uniformly $C^{r}$  along the leaves of $\w$.

\vskip.2cm

\noindent {\bf (4)}  For any  $x \in \M$ and $y \in \w_x$, the maps 
$\h_y \circ \h_x^{-1}$ and $\h_y' \circ (\h_x') ^{-1}$ from $T_{x}\w$ to $T_{y}\w$ are
compositions of a sub-resonance polynomial in $\s_{x,y}$ with a translation.

\vskip.2cm

\noindent {\bf (5)} The family $\{ \h_x \} _{x\in \M}$ as in (1) can be chosen so that $\,\h_x$  
which depends $C^{\lfloor r \rfloor}$ on $x$  along the leaves of $\w$.

\end{theorem}
 
We note that there is no analog of (5) for the resonance case. Also, the transition maps
$\h_y' \circ (\h_x') ^{-1}$ in (4) are only sub-resonance, so this is just a particular case of  
the result for $\h_y \circ (\h_x) ^{-1}$.
 
 Another way to interpret (4) is to view $\h_x$ as a coordinate chart on 
$\w_x$, identifying it with $T_{x}\w$, and in particular  identifying $T_{y}\w$ with 
$T_{\h_x(y)}(T_{x}\w)$ by $D_y\h_x $.  In this coordinate chart,  
(4) yields that all transition maps $\h_y \circ \h_x^{-1}$ for $y\in \w_x$ are  in the group 
$\bar \s_x$ generated by $ \s_x$ and the translations of $T_{x}\w$.  Clearly, this 
 is a finite dimensional Lie group which acts transitively on $T_{x}\w$.

\begin{remark} \label{smooth flag}
The diffeomorphism $\h_x$ maps the sub-foliations of $\w_x$ by fast leaves tangent to the fast 
flag \eqref{fastflag} to the linear sub-foliations of $T_{x}\w$ by subspaces parallel to the flag,
and the compositions $\h_y \circ \h_x^{-1}$ map flag to flag, see \cite[Section 3.2]{KS15}. 
It follows that these fast sub-foliations of $\w_x$ are as smooth inside $\w_x$ as $\h_x$.
While smoothness of fast sub-foliations is a well-known phenomenon, normal form results
give an alternative proof.
\end{remark} 

%Corollary: if leaves of an inv foliation are suff smooth then they are subres?

 \begin{corollary} \label{cor1/2pinch}
 Under the assumptions of the Theorem \ref{NFfol}, if $\ell=1$, i.e. $\chi_1=\chi_\ell =\chi$, then 
 $\p_x=Df|_{T_x\w}$ in (1) and  $Q_x=Dg|_{T_x\w}$ in (3) are linear normal forms, the family 
 $\{ \h_x \} _{x\in X}$ as in (1) is unique, the maps $\h_y \circ \h_x^{-1} : \E_x \to \E_y$ are affine 
 for all $x \in X$ and $y \in W_x$, and $\h_y$ depends $C^{\lfloor r \rfloor}$-smoothly on $y$ along 
 the leaves of $\w$.

In this cans we can take $\e<  -\a \chi/( 2+\a)$, where $\a=\min \{1, r-1\}$. This means that $Df|_{T\w}$ is a contraction 
whose characteristic set is contained in an open interval with ratio of endpoints at most $1+\a$. 
This is the ``constant rate" version of $1/2$ pinching \eqref{12pinch} (cf. Remark \ref{LinearR}).

 \end{corollary}

%%%%%%%%%%%%%%%%%%%%%%%%%%%%%%%
%%%%%%%%   Preliminaries %%%%%%%%%%%
%%%%%%%%%%%%%%%%%%%%%%%%%%%

%%%%%%%%%%%%%%%%%%%%%%%%%%%%%%%%%%%%
%%%%  PROOFS %%%%%%%%%%
%%%%%%%%%%%%%%%%%%%%%%%%%%%%%%%%%%%%

\section{Proof of Theorem \ref{NFext} and Corollary \ref{Cinf}} 

%? We give the proof for the case $\a>0$. The proof for $\a=0$, and hence $N\ge2$, is similar but avoids the need of estimating the H\"older constant at $0$. 
%We will use notation $\,\f_x^n = \f_{f^{n-1}x} \circ \dots \circ \f_{fx}\circ \f_x$.
We begin with the proof of part (1). First we construct Taylor polynomials at zero of degree $N$ 
for $\h_x$ and the corresponding terms of sub-resonance polynomials $\p_x$. Since \eqref{crit} 
implies $N\ge d= \lfloor \chi_1/\chi_\ell \rfloor$,
this will fully define $\p_x$, whose degree is at most $d$.  

\subsection{Construction of $\p$ and of the Taylor polynomial for $\h$}$\;$
For each $x\in X$ and map $\f_x : \E_{x} \to \E_{fx}$ we consider its 
Taylor polynomial  of degree $N$ at $t=0$: 
\begin{equation}\label{f_x}
\f_x(t) \sim \sum_{n=1}^N \fd^{(n)}_x(t).
\end{equation}
As a function of $t$, $\fd^{(n)}_x(t) :\E_{x} \to \E_{fx}$ is a  homogeneous polynomial 
map of degree $n$.  We will view the family $F ^{(n)}= \{ F^{(n)}_x\}_{x\in X}$ as a section of the
corresponding bundle of homogeneous polynomials.% and equip it with "sup norm".
We will use similar notations for the Taylor polynomials at $t=0$ of the desired coordinate 
change $\h_x(t)$ and the corresponding sub-resonance polynomial extension $\p_x(t)$:
$$\h_x(t)\sim \sum_{n=1}^N \hd^{(n)}_x(t)\quad \text{and} \quad 
 \p_x(t)= \sum_{n=1}^d \pd^{(n)}_x(t).
 $$
We will inductively construct the Taylor terms $H ^{(n)}= \{ H^{(n)}_x\}_{x\in X}$ and 
$P ^{(n)}= \{ P^{(n)}_x\}_{x\in X}$ as continuous sections of the
corresponding bundle of homogeneous polynomials.

For the first derivative we define 
$$\hd^{(1)}_x= \Id : \E_x \to \E_x \quad \text{and} \quad 
\pd^{(1)} _x =\fd _x  \quad \text{for all } x \in X.
$$
Now we assume that the terms of order less than $n$ are constructed.
Using these linear terms in the conjugacy equation 
$ \h_{fx}  \circ \f_{x} =\p_{x} \circ \h_{x} $ we write
$$
\left( \Id   +\sum_{i=2}^N  H^{(i)} _{fx}  \right) \circ
\left( F_x+\sum_{i=2}^N F^{(i)} _{x} \right)  
\sim \left( F_{x} + \sum_{i=2}^d P^{(i)}_{x} \right) \circ
\left( \Id +\sum_{i=2}^N H^{(i)}_{x} \right).
$$
and considering the terms of degree $n$, $2 \le n \le N$,  we obtain  
$$
F^{(n)}_{x}\, + \, H^{(n)}_{fx} \circ F(x)\,+ \,
\sum H^{(i)}_{fx} \circ F^{(j)}_{x}
\,= \,F_{x} \circ H^{(n)}_{x}+ P^{(n)}_{x}+\,
\sum P^{(j)}_{x} \circ H^{(i)}_{x},
$$
where the summations are over all $i$ and $ j$ such that $i j=n$ and $ 1<i,j<n$.
We rewrite the equation as 
\begin{equation}\label{Pn}
F_{x}^{-1} \circ P^{(n)}_{x} =  - H^{(n)}_{x}  + F_{x}^{-1} \circ H^{(n)}_{fx} \circ F_{x} + Q_{x}, \;
\end{equation}
where
\begin{equation}\label{Q}
Q_{x}= F_{x}^{-1} \left( F^{(n)}_{x} + 
\sum_{ij=n, \;\, 1<i,j<n} H^{(i)}_{fx} \circ F^{(j)}_{x}  
-  P^{(j)}_{x} \circ H^{(i)}_{x} \right).
\end{equation}
We note that $Q_x$ is composed only of terms $H^{(i)}$ and  $P^{(i)}$ with $1<i<n$, 
which are already constructed, and terms $F^{(i)}$ with $1<i\le n$,
which are given. Thus by the inductive assumption $Q_x$ is defined for all
 $x \in X$ and is continuous in $x$. 
 
Let $\po_x^{(n)}$ be  the space of all homogeneous polynomial maps on  $\E_x$ of 
degree $n$, and  let $\s_x^{(n)}$ and $\n_x^{(n)}$ be the subspaces of 
sub-resonance and non sub-resonance  polynomials  respectively.
We seek $H^{(n)}_{x}$ so that the right side of \eqref{Pn}
 is in $\s_x^{(n)}$, and hence so is $P^{(n)}_{x}$ when defined by this equation.

Projecting  \eqref{Pn} to the factor bundle $\po^{(n)} / \s^{(n)}$, our goal is to solve the  equation
 \begin{equation}\label{barHn}
0 =  - \bar H^{(n)}_{x}  + F_{x}^{-1} \circ \bar H^{(n)}_{fx} \circ F_{x} + \bar Q_{x}, \;
\end{equation}
where $\bar H^{(n)}$ and $ \bar Q$ are the projections of $H^{(n)}$ and $Q$
respectively.

We consider the bundle automorphism $\Phi : \po^{(n)} \to \po^{(n)}$ covering 
$f^{-1}: \M \to \M$ given by
the maps $\Phi_x : \po^{(n)} _{fx} \to \po^{(n)} _{x}$
 \begin{equation}\label{Phi}
\Phi_x (R)= \fd _{x}^{-1} \circ R  \circ \fd _{x}.
 \end{equation}
Since $\fd$ preserves the splitting $\E=\E^1\oplus \dots \oplus \E^{\ell}$, 
it follows from the definition that the sub-bundles $\s^{(n)}$ and $\n^{(n)}$ are $\Phi $-invariant. 
We denote by $\bar \Phi$ the induced automorphism of $\po^{(n)} / \s^{(n)}$
and conclude that \eqref{barHn} is equivalent to 
 \begin{equation}\label{fixedP}
 \bar H^{(n)}_{x} =  \tilde \Phi_x ( \bar H^{(n)}_{fx} ), 
 \quad \text{where }\; \tilde \Phi_x (R)= \bar \Phi_x (R) +  \bar Q_{x}.
\end{equation}

Thus a solution of \eqref{barHn} is a $\tilde \Phi$-invariant section of $\po^{(n)} / \s^{(n)}$. In Lemma \ref{contraction} below we will show that $\tilde \Phi$ is a contraction and hence has a unique continuous invariant section, which can be explicitly written as
%Then $T$ is also a contraction and thus has a unique fixed point 
\begin{equation}\label{bar s"}
\bar H _x  = \sum _{k=0}^\infty (F^k_{x})^{-1} \circ \bar Q_{f^k x} \circ  F^k_{x},
\quad \text{where }\,F^k_{x}= F_{f^{k-1}x}\circ \dots \circ F_{fx} \circ F_{x}.
\end{equation}

We equipped $\E$ with a continuous Riemannian metric as in Remark \ref{assSp} 
for which the splitting $\E=\E^{1} \oplus \dots \oplus \E^{l}$ is orthogonal and so that
 \eqref{estAEi'} and hence \eqref{estAnorm} hold.

The norm of a homogeneous polynomial map $R: \E_x \to  \E_y$ of degree $n$ is defined as 
\begin{equation} \label{normDef}
\|R\|=\sup\{\,\|R(v)\|:\; v\in E, \;  \|v\|=1 \,\}.
\end{equation}
 It follows that for any other homogeneous 
polynomial map $P: \E_z \to  \E_x$ we have
\begin{equation}\label{normP}
 \|\, R \circ P \,\|\leq \|R\|\cdot \|P\|^n.
\end{equation}

First, we look at the action of $\Phi $ on polynomials of specific homogeneous type.

\begin{lemma} \label{exponents}
Let $Q\in \po^{(n)} _{x} $ and $R \in \po^{(n)} _{fx} $ be polynomials of homogeneous type  
$s=(s_1,  \dots , s_{\ell})$ with $s_1+\cdots +s_\ell=n$. Then 
\begin{equation}
\| \Phi_x (R) \| \le e^{ -\chi_i+\sum s_j \chi_j +(n+1)\e} \, \| R \| \; \, \text{ and } \; \, \| \Phi^{-1}_x (Q) \|  \le   e^{\chi_i - \sum s_j \chi_j + (n+1)\e} \| R \|.
\end{equation}
\end{lemma}

\begin{proof}
We will prove the first inequality, the second one is obtained similarly.
Suppose that $v=v_1+ \dots +v_{\ell}$, where $v_j \in \E^j_x$, and $\|v\|_x=1$. 
We denote $a_j=\| F|_{\E^j_x}\|$ and observe that
$F_x(v_j)= a_j v_j'  \in \E^j_{fx}$ with $\|v_j'\|\le \|v_j\|$.
Since $R$ has homogeneous type $s= (s_1, \dots , s_\ell)$ we obtain 
by \eqref{stype} that
\begin{equation}
(R \circ F_x)(v)=R(a_1 v_1'+ \dots + a_\ell v_\ell ')= 
a_1^{s_1} \cdots  a_\ell^{s_\ell} \,\cdot R(v_1'+ \dots + v_\ell ').
\end{equation}
where $v'=v_1'+ \dots + v_\ell '$ has $\|v'\|_{fx}\le \|v \|_x=1$
by orthogonality of the splitting.
Thus 
$$
\| (R \circ F_x)(v)\|= a_1^{s_1} \cdots  a_\ell^{s_\ell} \cdot
\| R(v')\| \le a_1^{s_1} \cdots  a_\ell^{s_\ell} \,\cdot
\| R\|
$$
for any $v \in \E_x$ with $\|v\|_x=1$, so  by definition  \eqref{normDef} we obtain
$\| R \circ F_x\|\le a_1^{s_1} \cdots  a_\ell^{s_\ell} \,
\cdot \| R\| $.
Now \eqref{normP} yields
$$
\begin{aligned}
&\| \Phi_x (R) \| =\|F|_{\E^i_x}^{-1} \circ R \circ F_x\| \le \|F|_{\E^i_x}^{-1}\| \cdot 
\| R \circ F_x\|\le \\
&\le \|F|_{\E^i_x}^{-1}\| \cdot a_1^{s_1} \cdots  a_\ell^{s_\ell} \cdot
\| R\| \le e^{-\chi_i +\e}  \cdot \prod_j (e^{\chi_j +\e})^{s_j} 
 \cdot  \| R \| .
 \end{aligned}
$$ 
Since $a_j=\| F|_{\E^j_x}\|\le e^{\chi_j +\e}$ and 
$\|F|_{\E^i_x}^{-1}\| \le e^{-\chi_i +\e}$ by \eqref{estAnorm}.
\end{proof}

\begin{lemma} \label{contraction}
The map $\Phi : \n^{(n)} \to \n^{(n)}$ given by \eqref{Phi} is a  contraction over $f^{-1}$, and hence so is $\,\tilde \Phi : \po^{(n)} / \s^{(n)} \to \po^{(n)} / \s^{(n)}$  given by \eqref{fixedP}. More precisely,  
$\| \Phi_x (R) \| \le e^{\la +(d+2)\e}\cdot \| R \| $.
\end{lemma}
\begin{proof}
The statement about $\,\tilde \Phi$ follows since the linear part $\bar \Phi$ of $\tilde \Phi$ 
is given by $\Phi$ when $\po^{(n)} / \s^{(n)}$ is naturally identified with $\n^{(n)}$.

For all non sub-resonance homogeneous types we have $-\chi_i+\sum s_j \chi_j  \le \tilde \la $ by the
definition of $\tilde \lambda$ \eqref{lambda} and hence for any $R \in \n^{(n)} _{fx}$ Lemma \ref{exponents} 
yields the estimate $\| \Phi_x (R) \| \le e^{\tilde \la +(n+1)\e}\cdot \| R \| $. For all $n \le d$ the exponent satisfies 
$\tilde \la +(n+1)\e \le  \la +(d+1)\e <0$ since $\e<\e_0 \,$ given by \eqref{epsilon}.

If $d+1\le n\le N$, then $ \s^{(n)}=0$ and $\po^{(n)} =\n^{(n)}$. In this case for any $R \in \po^{(n)} _{fx}$ 
we can estimate $\| \Phi_x (R) \| \le e^ {-\chi_1+n \chi_\ell +(n+1)\e}\cdot \| R \|$ and by the
definition of $\lambda$ \eqref{lambda} we have $-\chi_1+(d+1) \chi_\ell \le \lambda<0$ and hence
the exponent satisfies 
$$-\chi_1+n \chi_\ell +(n+1)\e \le  \la +(d+2)\e +(n-(d+1))(\e + \chi_\ell) <0$$
 since $\e<\e_0 \,$ given by \eqref{epsilon}.
\end{proof}

\vskip.1cm

We conclude that the unique continuous invariant section \eqref{bar s"} for $\tilde \Phi$ is the unique continuous solution $\bar H^{(n)}$ of \eqref{barHn}. Now can we choose a continuous section 
$H^{(n)}$ of $\po^{(n)}$ which projects to $\bar H^{(n)}$. Such $H^{(n)}$ is defined uniquely 
up to a continuous section $\s^{(n)}$. 
\begin{remark} \label{smooth dep}
For example one can take $H^{(n)}$ in $\n^{(n)}$. However, in the foliation setting the bundles 
$\po^{(n)}$, $\s^{(n)}$, their quotient, and $\bar H^{(n)}$ are often more regular along the leaves 
than $\n^{(n)}$. In this case one can make a more regular choice for $H^{(n)}$ leading to better
dependence of $\h_x$ on $x$ along the leaves. See \cite{KS15} for more details of this argument.
\end{remark}

Once a specific lift $H^{(n)}$ is chosen, $\pd^{(n)}_x$ is uniquely define by equation \eqref{Pn} 
and is a continuous section $\pd^{(n)}$. This completes the inductive step and the construction of
$\hd^{(n)}$ and $\pd^{(n)}$, $n=1, \dots, N$.
\vskip.2cm
Thus we have constructed the $N$-th Taylor polynomials for the coordinate changes
\begin{equation}\label{H^N}
\h_x^N(t)= \sum_{n=1}^N \hd^{(n)}_x(t)\quad\text{of  degree }\;
N\ge d=\lfloor \chi_1/\chi_{\ell} \rfloor
\end{equation}
 and the polynomial maps 
$\p_x(t)= \sum_{n=1}^d \pd^{(n)}_x(t)$.

%%%%%%%%%%%%%%%%%%%%%%%%%%%
%%%%%%%%%%%%%%%%%%%%%%%%%%%%%%%%%%%%%
\subsection{Construction of the coordinate change $\h$.} 
%%%%%%%%%%%%%%%%%%%%%%%%%%%

In this section we complete the proof of part (1) by constructing the actual coordinate changes 
$\h_x$ with the Taylor polynomial $\h_x^N$ given by \eqref{H^N}. 
To simplify the calculations we note that $\h^N_x(t)$ is a diffeomorphism on some neighborhood 
$\tilde \V_x$ of $0\in\E_x$ since its differential at $0$ is $\Id$, moreover the size of 
$\tilde \V_x$ can be bounded away from $0$ by compactness of $\M$. Thus we can consider extension 
$\tilde \f_x(t)=\ \h_{fx}^N \circ \f_{x} \circ ( \h_x^N)^{-1}$. By the construction of $ \h^N$, 
the maps $\tilde \f_x$ and $\p_x $ have the same derivatives at $t=0$ 
up to order $N$ for each $x \in \M$. Since the construction is done in a sufficiently small 
neighborhood of zero section, we can replace $\f$ by $\tilde \f_x$, and so henceforth we 
  assume that $\f$ itself has this property.

We rewrite the conjugacy equation 
$\h_{fx} \circ \f_x=\p_x\circ \h_x$ in the form 
\begin{equation} \label{T}
\h_x= \tilde T(\h)_x=\p_x^{-1} \circ \h_{fx} \circ \f_x.
\end{equation}
so that  solution $\h=\{ \h_x \}$ is a fixed point of the operator $\tilde T$. Since $\p_x$ is a 
sub-resonance polynomial with with invertible linear part $ \pd^{(1)}_x=F_x$, by the group property
 the inverse $\p_x^{-1}$ is also a sub-resonance polynomial and thus has degree at most $d\le N$.

Denoting $\bar H_x=\h_x-\Id$
we rewrite \eqref{T} as
\begin{equation} \label{tilde T}
\bar H_x =  T(\bar H)_x=  \p_x^{-1} \circ (\Id+\bar H_{fx}) \circ \f_x - \Id.
\end{equation}
Thus the coordinate change $\h$ corresponds to the fixed point $\bar H= T(\bar H)$
 of the operator $T$ on the space of continuous families of smooth functions. We will show that 
 $T$ is a contraction  on an appropriate space and thus has a unique fixed point. 

For any $x \in X$ we consider the ball $B_{x,r}$ in $\E_x$ centered at $0$ of radius $r$ and denote
  $$
  \c_{x} = \c_{x,r}= \{ R \in \Cr (B_{x,r},\E_x) : \;
  D^{(k)}_0 R =0, \;\, k=0,...,N \},
  $$
where $\Cr (B_{x,r},\E_x)$ and its norm are defined as in \eqref{Crnorm}. We note that for any 
$R\in \c_{x}$ the $\a$-H\"older constant \eqref{Canorm} of $D^{(N)} R$ at $0$ is 
\begin{equation}\label{norm a}
 \|D^{(N)} R\|_{\a}= \sup\, \{ \| D^{(N)}_t R\|\cdot \|t\|^{-\a} : \;\, 0\ne t \in B_{x,r}\}.
\end{equation} 
For any $R\in \c_{x}$  lower derivatives can be estimated by the mean value theorem as 
\begin{equation}\label{deriv'}
 \|D^{(n)}_t R \| \le \|t\|^{N-n} \cdot
 \sup \,\{ \|D^{(N)}_s R \| : \,\| s\| \le \|t\| \} ,
  \end{equation} 
so using the above H\"older constant we obtain that for any  $0 \le n<N$ 
and $t\in B_{x,r}$,
\begin{equation}\label{deriv}
 \|D^{(n)}_t R \|  \le \|t\|^{1+\a}\cdot \|D^{(N)} R\|_{\a}.
\end{equation} 
Thus for $r<1$ the norms of all derivatives are dominated by the H\"older constant and hence
\begin{equation}\label{Cr=a}
  \|R \|_{\Cr (B_{x,r},\E_x)} = \|D^{(N)} R \|_{\a}.
\end{equation} 
It follows that $\c_{x}$ equipped with  the norm $\|D^{(N)} R \|_{\a}$ is a Banach space.
We denote by $\c$ the bundle over $X$ with fibers $\c_x$ and by $\S$ the space of 
sections $\bar R=\{ R_x\}_{x\in X}$ of $\c$ which are bounded in $\Cr $ norm and 
continuous in $C^N$ norm. Then $\S$ is a Banach space with the norm  
$\| \bar R \|_{\c}= \sup_x \| R_x \|_{\a}$.

We consider $T$ as an operator on $\S$. It follows from  the definition of $\c_x$ and 
the coincidence of the derivatives of $\p_x$ and $\f_x$ at 0 that $T(\bar R)$ is in $\S$.
We will show that that, for a sufficiently small $r$, $T$ is a contraction on some ball 
$B_\gamma$ in $\S$.% centered at $0$ of radius $\gamma$. 
%%%%%%%%%%%
\vskip.2cm

We will now define the parameters for this argument.
First we note that $0<\e <-\chi_\ell \,$ since $\e<\e_0 \,$ given by \eqref{epsilon}, and that we have
$\nu-(N+1+\a) \e>0$ by assumption \eqref{crit} so we can take $\e'>0$ satisfying
\begin{equation}\label{delta}
\chi_\ell +\e+\e'<0  \quad \text{and } \;  \delta=\nu-(N+1+\a)(\e+\e')>0.
\end{equation}
We also recall that  by \eqref{estAnorm} we have
\begin{equation}\label{FP}
 D_0 \p_{x} = D_0\f_x =F_x,  
 \quad \|F_x\| \le e^{\chi_\ell+\e}, \quad \text{and } \; \| F_x^{-1} \| \le e^{-\chi_1+\e} 
\end{equation}
Now we can choose $\rho<\min\{1,\sigma\}$ sufficiently small so that for all $x \in X$ we have
\begin{equation}\label{Pt''}
\| D_{t} \, \f_{x} \| \le e^{\chi_\ell+\e+\e'}  \quad \text{and } \; \| D_t \, (\p_{x})^{-1}  \| \le e^{-\chi_1+\e+\e'} \;
\text{ for all $t \in B_{x,\rho}$},
\end{equation}
so in particular $\f_x:B_{x,\rho} \to B_{fx,\rho}$ is a contraction. We choose $K$ so that 
\begin{equation}\label{CFP}
  \|\f \|_{C^N (B_{x,\rho})} \le K  \quad \text{and } \; \| (\p_{x})^{-1}\|_{C^N (B_{x,\rho})} \le K \;
\text{ for all $x \in X$}.
\end{equation} 
Since $\delta$ given by \eqref{delta} is positive, we can define $\theta>0$ by
\begin{equation}\label{theta}
1-2\theta = \,e^{-\delta} <1 \quad \text{and let } \;  \gamma=\max \{1,\|T(\bar 0)\|_\S/\theta \},
 \end{equation}
here, as $r$ is not yet defined, we take $r=\rho$ in the definition of the norm $\|T(\bar 0)\|_\S$
(this does not create problems since the norm decreases with $r$).
To show that $T$ is a contraction on the ball $B_\gamma$ in $\S$ centered at $\bar 0$ of radius $\gamma$ 
we will  estimate the norm of its differential by $1-\theta$.  For this we choose $r>0$  satisfying 
\begin{equation}\label{r}
 r<\rho<1, \quad r<\rho /(1+\gamma), \quad r \le  \theta /(c_3(K,N)\, \gamma^N)
\end{equation} 
where constant $c_3(K,N)$ from \eqref{Jest} depends only on $N$ and $K$.
%%%%%%%%%%%
\vskip.2cm

Now we will calculate the differential of $T$ on $B_\gamma$ and estimate its norm.
For any $\bar R, \bar S \in B_\gamma$ we can write
$$
(T(\bar R+\bar S)-T( \bar R))_{x} =   (\p_{x})^{-1} \circ (\Id+R_{fx}+S_{fx}) \circ  \f_{x}
- (\p_{x})^{-1} \circ (\Id+R_{fx}) \circ  \f_{x}.
$$
Differentiating $(\p_{x})^{-1}$ and denoting 
$$
y(t)=(\Id+R_{fx})( \f_{x}(t)) =\f_{x}(t)+R_{fx}( \f_{x}(t)) \quad\text{and}\quad
z(t) = S_{fx}( \f_{x}(t))
$$ 
we obtain
$$
(T(\bar R+\bar S)-T(\bar R))_x(t)=   D_{y(t)} (\p_{x})^{-1} \, z(t)+ E(z(t)),
$$
where $E$ is a polynomial with terms of degree at least two. It follows 
that $\|E(z(t))\|_{\c} =O (\|\bar S\|^2_{\c})$ and so the differential of $T$ is given by
$$
([D_{\bar R} T] \bar S) _x(t)= D_{y(t)}(\p_{x})^{-1} \; S_{fx} ( \f_{x}(t))= A_x(y(t))  z(t),
$$
where $A_x(s)= D_{s} (\p_{x})^{-1}$.
To estimate the norm we consider the derivative of order $N$.

Since  $A_x(y(t))$ is a linear operator on $z$, the product rule yields
 \begin{equation} \label{DTN}
D^{(N)} [A_x(y(t))  z(t)]= A_x(y(t)) D^{(N)} z(t) + \sum c_{m,l}\, D^{(m)}A_x(y(t))  D^{(l)} z(t),
\end{equation}
where $m+l=N$ and $l<N$ for all terms in the sum. Differentiating $z(t)$ we get
$$
D^{(l)} z(t) = D^{(l)}  S_{fx}( \f_{x}(t))=  \sum D_{t'}^{(i)} S_{fx} \circ D_t^{(j)} \f_{x_k}, 
$$
where $ij=l$ and $t'= \f_{x}(t)$. Only the first term in \eqref{DTN} contains $D^{(N)} S_{fx}$ so 
 \begin{equation} \label{est}
D_t^{(N)} ([D_{\bar R} T] \bar S) _x = D_{y(t)}^{(1)} \, (\p_{x_k})^{-1} \circ  D_{t'}^{(N)} S_{fx} \circ D_t^{(1)} \f_{x} \,+ \,J_x,
\end{equation}
where  $J_x$ consists of  a fixed number $c(N)$ of terms of the type 
$$
   D_{t}^{(m)} A_x(y(t))  \left(D_{t'}^{(i)} S_{fx} \circ D_t^{(j)} \f_{x} \right), 
   \quad  i<N , \; m+ij=N,
$$
whose norms  can be estimated by
 \begin{equation} \label{J}
 \| A_x(y(t)) \|_{C^N} \cdot \| D_{t'}^{(i)} S_{fx} \|\cdot \| \f_{x}\|^{N-1}_{{C^N}}.
\end{equation}
We start with the middle term and observe that  
 \begin{equation} \label{t'}
\|t'\|=\|\f_{x}(t))\| \le e^{\chi_\ell +\e+\e'} \|t\|<  \|t\| \le r
\end{equation}
 by \eqref{Pt''}. Since $i<N$, we can estimate 
the middle term using \eqref{deriv} 
\begin{equation} \label{middle}
 \|  D_{t'}^{(i)} S_{fx}\| \le 
  \|t'\|^{1+\a} \cdot \| D^{(N)} S_{fx} \|_{\a} 
 < \|t\|^{1+\a} \cdot \|\bar S \|_{\c} \le r \|t\|^{\a} \cdot \|\bar S \|_{\c}.
\end{equation}
For the first term we note that $t'' =y(t)\in B_{x,\rho}$, indeed since $\| \bar R \|_{\c} \le \gamma$ 
we get
 \begin{equation} \label{t''}
\|t'' \|=\| y(t)\|=\|(\Id+R_{fx})( \f_{x}(t))\|=\|t' + R_{fx}(t')\| < (r+ \gamma r) <\rho
\end{equation}
 by \eqref{t'} and the choice of $r$ \eqref{r}. We also use an estimate for the norm of composition 
of smooth maps, see e.g  \cite{dlLO}:
\begin{lemma}  \cite[Theorem 4.3(ii.3)]{dlLO} \label{LL}
For any $N\ge 1$ there exist a constant $M_N$ such that
$\| h\circ g\|_{C^N}  \le M_N \,  \|h\|_{C^N} (1+\| g\|_{C^N})^N$
 for any $h,g \in C^N(B_{x,\rho})$.
\end{lemma}
Using this together with  \eqref{CFP}, $\| \bar R \|_{\c} \le \gamma$, and $\gamma \ge 1$ and we can write
$$  
 \| y(t)\| _{C^N} = \| (\Id+R_{fx}) \circ \f_{x} \| _{C^N} \, \le  M_N \, (1+\gamma)\, (1+K)^N \le c_1(K,N)\gamma ,
  $$
where $c_1(K,N)=2M_N (1+K)^N$.
Now we can estimate the first term in \eqref{J}   
$$
 \| A_x(y(t)) \|_{C^N} \, \le \, M_N \, \| (\p_{x})^{-1} \|_{C^N}  \| y(t) \| _{C^N}^N \,
 \le M_N \, K [1+c_1(K,N)\gamma ]^N \le c_2(K,N)\gamma^N,
$$
where $c_2(K,N)=M_N \, K (1+c_1(K,N))^N$, and we note that $\p_{x}^{-1}$ is a polynomial of degree at most $d\le N$.  Combining this with \eqref{middle} and 
$K^{N-1}$ estimate from  \eqref{CFP} for the last term in \eqref{J}, we obtain the following 
estimate for the norm of $J_x$ in \eqref{est}:
\begin{equation}\label{Jest}
\| J_x\| \le c(N) c_2(K,N)\gamma^N \cdot  r \, \|t\|^{\a}  \|\bar S \|_{\c} \cdot K^{N-1}
\le  c_3(K,N)\, \gamma^N \cdot r \, \|t\|^{\a} \cdot \|\bar S \|_{\c}.
\end{equation}
where $c_3(K,N)=c(N) c_2(K,N) K^{N-1}$ and $c(N)$ is an estimate on the number of terms in $J_x$.

%%%%%%%%%%%%%%%%%%%%%%%%%%%%%%%

\vskip.1cm

Now we estimate  the main term in \eqref{est} using  \eqref{norm a},  \eqref{Pt''},  \eqref{t''},  and \eqref{t'}:
\begin{equation} \label{main term}
\begin{aligned}
&\| D_{t''}^{(1)} \, (\p_{x})^{-1} \circ D_{t'}^{(N)} S_{fx} \circ D_t^{(1)} \f_{x} \| \,\le \\
&\le 
  \| D_{t''}^{(1)} (\p_{x})^{-1} \| \cdot \|D^{(N)} S_{fx} \|_{\a}\, \|t'\| ^\a \cdot \|  D_t^{(1)} \f_{x} \| ^N \,\le \\
 & \le  e^ {-\chi_1+\e+\e' } \cdot \|\bar S \|_{\c}  \cdot e^{\a (\chi_\ell+\e+\e')}
 \|t\|^\a   \cdot e^{N (\chi_\ell+\e+\e')}  =  e^{-\delta} \|t\|^\a\, \| \bar S \|_{\c} ,
 \end{aligned}
\end{equation}
by the definition of  $\delta$  \eqref{delta} where $\nu=-(N+\a)\chi_\ell+\chi_1$.
By  \eqref{theta} we have $ e^{-\delta} =1-2\theta$.
Finally we estimate \eqref{est} combining \eqref{Jest} and \eqref{main term}.  
For any $\bar R \in B_\gamma$ 
$$
\|D_t^{(N)} ([D_{\bar R} T] \bar S) _x  \| \le \,
 \|t\|^\a \cdot \| \bar S \|_{\c} 
 \left( 1-2\theta +  c_3(K,N)\, \gamma^N \cdot r  \right) \le \,
 \|t\|^\a \cdot \| \bar S \|_{\c} 
 \left( 1-\theta \right).
$$
since $r \le  \theta /(c_3(K,N)\, \gamma^N)$ by \eqref{r}.
Then for all $\bar R \in B_\gamma$ we obtain
$$
\begin{aligned}
& \|\,D_t^{(N)} ([D_{\bar R} T] \bar S) _x  \,\| \le \,
 \|t\|^\a \cdot \| \bar S \|_{\c} 
 \left( 1-\theta  \right), \quad\text{hence}\\
&\|\,D^{(N)} ([D_{\bar R} T] \bar S) _x  
 \,\|_{\a} \le \,  (1-\theta) \cdot  \| \bar S \|_{\c} ,
\quad \text{and so}\\
&\|\,[D_{\bar R} T] \bar S \,\|_{\c} =  \,
\sup_x \,  \,\|\,D^{(N)}  (T(\bar S))_x \,\|_{\a} 
\le \, (1-\theta) \cdot \| \bar S \|_{\c} .
\end{aligned}
$$
Thus $\|D_{\bar R} T \|\le 1-\theta$ for all $\bar R \in B_\gamma$. 
Since $\| T(\bar 0)\|_{\c} \le  \theta \gamma$ from the definition of $\gamma$ \eqref{theta},
the operator $T$ is a contraction from $B_\gamma$ to itself. Thus $T$ has a unique fixed point 
$\bar H \in B_\gamma$, which is section of $\c$ bounded in $\Cr $ norm and continuous in $C^N$ 
norm. The corresponding family of $\Cr$ maps $ \h_x=\Id+H_x$
satisfies \eqref{T}, i.e. conjugates $\p_x$ and $\f_x$. Then the maps  $\h_x$ defined on $B_{x,r}$ 
can be uniquely extended to $\Cr$ diffeomorphisms on $B_{x,\rho}$, and then on $B_{x,\sigma}$, 
by the invariance
 $$
\h_x (t)= (\p^k_x)^{-1} \circ \h_{f^kx} \circ \f^k_x (t)
$$
since for each $t \in B_{x,\sigma}$ we have $\f^k_x (t) \in B_{x,r}$ for  some $k$.

%%%%%%%%%%%%%%%%%%%%%%%%%%%%%%%%%%%%%%
%%%%%%%%%%%%%%%%%%%%%%%%%%%%%%%%%%%%%%

\subsection {Prove of part (2): the (non)uniqueness of $\h$ and $\p$}
This essentially follows from the ``uniqueness" of the construction.
First, starting with $\h_1 =\tilde \h$ we inductively construct coordinate changes $\h_k=\{ \h_{k,x}\}$ 
for $k=1,..., N$ so that the corresponding normal form  $\p_{k ,x}$ is of sub-resonance type and their 
Taylor polynomials $\h_{k,x}(t) \sim \sum_{n=1}^N \hd^{(n)}_{k,x}(t)$ coincide with the Taylor polynomial 
of $\h$ to order $k$, that is $\hd^{(n)}_{x}=\hd^{(n)}_{k,x}$ for $n=1,...,k$. For $\h_1 =\tilde \h$ we have 
$\hd^{(1)}_{x}=\hd^{(1)}_{1,x}=\Id$ and $\p_{1 ,x}$ is sub-resonance  by the assumption.

Suppose $\h_{k-1}$, $k\ge 2$, is constructed and has $\hd^{(n)}_{x}=\hd^{(n)}_{k-1,x}$ for $n=1,...,k-1$.
Then $\p$ and $\p_{k-1}$  have the same terms up to order $k-1$. Hence $\hd^{(k)}_{k-1,x}$ and 
$\hd^{(k)}_{x}$ satisfy the same equation \eqref{barHn} when projected to the factor-bundle 
$\po^{(k)} / \s^{(k)}$. Indeed, the $Q$ term defined by \eqref{Q} is composed only of $F^{(i)}$ and
 terms $H^{(i)}$ and  $P^{(i)}$ with $1<i\le k-1$, which are the same for $\h_{k-1}$ and $\h$.
By uniqueness we obtain that 
$$
\hd^{(k)}_{x}=\hd^{(k)}_{k-1,x} + \Delta^{(k)}_x,\,\text{ where }\,
\Delta^{(k)}_x \in \s^{(k)}_x.
$$
Then the coordinate change $\h_{k,x}=(\Id + \Delta^{(k)}_x)\circ \h_{k-1,x}$ has the same Taylor terms 
as $\h$ up to order $k$ and, since the polynomial $\Id + \Delta^{(k)}_x$ is in $\s_x$, $\h_{k}$ 
conjugates $\f$ to a sub-resonance normal form 
$\p_{k,x}= (\Id + \Delta^{(k)}_{fx})\circ\p_{k-1,x} \circ (\Id + \Delta^{(k)}_x)^{-1}$.

Thus in $N$ steps we obtain the coordinate change   
$$
\h_{N,x}= G_x \circ \tilde \h_{x}, \;\text{ where }\,
G_x = (\Id + \Delta^{(N)}_x) \circ \dots \circ (\Id + \Delta^{(2)}_x) \in \s_x,
$$
 which has the same Taylor terms at $0$ as $\h$ up to order $N$.
 In fact,  for $n>d$ we have $\s^{(n)}=0$ and hence $\Delta^{(n)}=0$, so that $\h_{N}=\h_{d}$. 

Finally, we conclude $\h_{d}=\h_{N}=\h$ by the ``moreover part of (2), which follows from
the uniqueness of the fixed point in the last step of the construction.

%%%%%%%%%%%%%%%%%%%%%%%%%%%%%%%%%%%%%
%%%%%%%%%%%%%%%%%%%%%%%%%%%%%%%%%%%%%
\subsection{Prove of part (1'): construction of resonance normal form}$\;$

Now we construct a {\em polynomial} coordinate change $\h$ that brings the sub-resonance normal form
$\p_x= \sum_{n=1}^d \pd^{(n)}_x$ to a resonance normal form $\tp_x= \sum_{n=1}^d \tpd^{(n)}_x$.

We will inductively construct the terms of polynomial coordinate changes $\h_x =\sum_{n=1}^d \hd^{(n)}_x$. 
All terms will be in the group $\s$ of sub-resonance polynomials, so the process will stop in $d$ steps 
and yield sub-resonance of polynomial diffeomorphisms $\h_x$ and resonance normal form $\tp_x$.
The base case is $n=1$, where we take $\hd^{(1)}_x=\Id$, which leaves $\tpd^{(1)}_x= \pd^{(1)}_x=F_x \,$ in the resonance form, i.e. block triangular. 

Now we assume inductively that the terms of degree $k<n$ are constructed so that $\hd^{(k)}$ 
is a continuous section of  $\s^{(k)}$ and  $\tpd^{(k)}_x$  is a resonance polynomial and is continuous in $x$.  
As before, we consider the terms of degree $n$ in the conjugacy equation 
$ \h_{fx}  \circ \p_{x} =\tp_{x} \circ \h_{x} $ 
$$
\pd^{(n)}_{x}\, + \, H^{(n)}_{fx} \circ F_x\,+ \,
\sum H^{(i)}_{fx} \circ P^{(j)}_{x}
\,= \,F_{x} \circ H^{(n)}_{x}+ \tpd^{(n)}_{x}+\,
\sum \tpd^{(j)}_{x} \circ H^{(i)}_{x},
$$
where the summations are over all $i$ and $ j$ such that $i j=n$ and $ 1<i,j<n$. Then 
\begin{equation}\label{tPn}
 \tpd^{(n)}_{x} \circ F_{x}^{-1} =  - F_{x} \circ H^{(n)}_{x} \circ F_{x}^{-1}  +  H^{(n)}_{fx}  + Q_{x}, \;
  \text{ where}
\end{equation}
$$
Q_{x}=  \left( \pd^{(n)}_{x} + 
\sum_{ij=n, \;\, 1<i,j<n} H^{(i)}_{fx} \circ P^{(j)}_{x}  
-  \tpd^{(j)}_{x} \circ H^{(i)}_{x} \right) F_{x}^{-1}.
$$
We note that $Q_x$ is composed only of terms $H^{(i)}$ and  $\tpd^{(i)}$ with $1<i<n$, 
which are already constructed, and terms $P^{(i)}$ with $1<i\le n$, which are given. 
Since composition of sub-resonance polynomials is a sub-resonance polynomial, 
by the inductive assumption $Q_x$ is a continuous section of  $\s^{(n)}$. 
 
We consider splitting $\s_x^{(n)}=\r^{(n)}_x \oplus \ss_x^{(n)}$, where $\r^{(n)}_x$ and $ \ss_x^{(n)}$ 
denote the maps of resonance and strict sub-resonance types respectively.
We seek $H^{(n)}_{x}$ so that the right side of \eqref{tPn}
 is in $\r_x^{(n)}$, and hence so will be $\tpd^{(n)}_{x}$ when defined by this equation.

Projecting  \eqref{tPn} to the factor bundle $\s^{(n)} / \r^{(n)}$ we need to solve the  equation
 \begin{equation}\label{tbarHn}
0 =  - F_{x} \circ \bar H^{(n)}_{x} \circ F_{x}^{-1} +  \bar H^{(n)}_{fx}  + \bar Q_{x}, \;
\end{equation}
where $\bar H^{(n)}$ and $ \bar Q$ are the projections of $H^{(n)}$ and $Q$
respectively. We consider the automorphism $\Phi ^{-1}$ of the bundle $\s^{(n)}$  covering $f$ 
with fiber maps 
 \begin{equation}\label{PhiR}
\Phi_x^{-1} : \s^{(n)} _{x} \to \s^{(n)} _{fx} \quad \text{where }\;  \Phi_x ^{-1}(R)= \fd _{x} \circ R  \circ \fd _{x}^{-1}.
\end{equation}
Since $\fd$ preserves the splitting $\E=\E^1\oplus \dots \oplus \E^{\ell}$, 
the resonance and  strict sub-resonance types are preserved by $\Phi ^{-1}$. 
We denote by $\bar \Phi^{-1}$ the induced automorphism of $\s^{(n)} / \r^{(n)}$
and see that \eqref{tbarHn} becomes 
 \begin{equation}\label{tfixedPR}
 \bar H^{(n)}_{fx} =  \tilde \Phi_x^{-1} ( \bar H^{(n)}_{x} ), 
 \quad \text{where }\; \tilde \Phi_x^{-1} (R)= \bar \Phi^{-1}_x (R) -  \bar Q_{x}.
\end{equation}
Thus a solution of \eqref{tbarHn} is a $\tilde \Phi^{-1}$-invariant section of $\s^{(n)} / \r^{(n)}$. 

\begin{lemma} \label{contractionR}
The map $\Phi ^{-1}: \ss^{(n)} \to \ss^{(n)}$ given by \eqref{PhiR} is a  contraction over $f$, and hence so is $\,\tilde \Phi^{-1} : \s^{(n)} / \r^{(n)} \to \s^{(n)} / \r^{(n)}$  given by \eqref{tfixedPR}. More precisely,  
$\| \Phi_x^{-1} (R) \| \le e^{\la +(d+2)\e}\cdot \| R \| $.
\end{lemma}
\begin{proof}
The statement about $\,\tilde \Phi^{-1}$ follows since the linear part $\bar \Phi^{-1}$ of $\tilde \Phi^{-1}$ 
is given by $\Phi^{-1}$ when $\s^{(n)} / \r^{(n)}$ is naturally identified with $\ss^{(n)}$. 

By Lemma \ref{exponents}, for polynomials of homogeneous type  $s=(s_1,  \dots , s_{\ell})$ with 
$s_1+\cdots +s_\ell=n$ we have $\| \Phi^{-1}_x (R) \|  \le   e^{\chi_i - \sum s_j \chi_j + (n+1)\e} \| R \|$.
For all strict sub-resonance homogeneous types we have $\chi_i - \sum_{j=1}^\ell s_j \chi_j \le \mu $ by the
definition of $\mu$ \eqref{mu} and hence for any $R \in \ss^{(n)} _{x}$ we have 
$\| \Phi_x (R) \| \le e^{\mu +(n+1)\e}\cdot \| R \| $. Since $n \le d$ the exponent satisfies 
$\mu +(n+1)\e \le  \la +(d+1)\e <0$ since $\e<\e_0 \,$ given by \eqref{epsilon}.
\end{proof}

We conclude that $\tilde \Phi^{-1}$ is a contraction and hence has a unique continuous invariant section
$\bar H^{(n)}$. We choose a continuous section $H^{(n)}$ of $ \s^{(n)}$ which projects to $\bar H^{(n)}$, which is defined uniquely up to a section of $\r^{(n)}$.  For example one can take $H^{(n)}$ in $\ss^{(n)}$.
 Once $H^{(n)}$ is chosen, we define $\tpd^{(n)}_x$ by equation \eqref{tPn} and get a continuous section $\tpd^{(n)}$ of  $\r^{(n)}$. This completes the inductive step and the construction of
$\h$ and $\tp$.

%%%%%%%%%%%%%%%%%%%%%%%%%%%%%%%%%%%%%
%%%%%%%%%%%%%%%%%%%%%%%%%%%%%%%%%%%%%
\subsection{Prove of part (2'): the (non)uniqueness for resonance normal form}$\;$

This follows from the ``uniqueness" of the construction in the previous section similarly 
to the proof of  part (2). The process of transition from  $\tilde \h'_x$ to  $\h'_x$ stays in the
group of resonance polynomials $\r_x$ and we obtain  
$\h_x=\h_{d,x}= G_x \circ \tilde \h_{x}'$, where $G_x  \in \r_x$. 

%%%%%%%%%%%%%%%%%%%%%%%%%%%%%%%%%%
%%%%%%%%%%%%%%%%%%%%%%%%%%%%%%%%%%%%%

\subsection {Proof of part (3): centralizer.}

First we prove that the derivative of $\g$ at zero section, $\G_x=D_0\g_x$, is sub-resonance. 
Since $\G_x$ is linear, this is equivalent to the fact that $\G_x$ preserves the flag of fast sub-bundles
associated with the splitting \eqref{fastflag}.
Suppose to the contrary that for some $x \in X$ and some $i>j$ we have a unit vector $t$ in $\E^j_x$ 
such that $t'=\G_x(t)$ has nonzero component $t'_i \ne 0$ in $\E^i_{gx}$. Then we have
$$
\| (F ^n_{gx}  \circ \G_{x} )(t) \| \ge \| F ^n_{gx} (t'_i) \| \ge e^{(\chi_i -\e)n}\, \| t'_i \|.
$$
On the other hand, since the extensions and hence their derivatives commute, we have 
$$
\| (F ^n_{gx}  \circ \G_{x} )(t) \| = \| \G_{f^n x} ( F ^n_x (t)) \| \le 
\| \G_{f^n x} \| \cdot e^{(\chi_j+\e)n} \| t \| \le C e^{(\chi_j +\e)n},
$$
which is impossible for large $n$ as $\e$ is small enough. Indeed, since $i>j$ we have non sub-re\-sonance relation $-\chi_i + \chi_j  <0$, so by definition \eqref{lambda}
of $\tilde \la$ we have  $-\chi_i + \chi_j \le \tilde \la <0$, and hence by definition 
\eqref{epsilon} of $\e_0$ we have $\e_0 \le -\lambda/(d+2) \le  -\tilde \lambda/3<(\chi_i - \chi_j)/2$.
Since $\e < \e_0$, this yields $ \chi_j +\e <\chi_i -\e$.

Similarly, we can further show that $\G_x$ is resonance, i.e. preserves the splitting.
Using notations above, if $i<j$ we can estimate backward iterates: for $n<0$ we have 
$$
e^{(\chi_i +\e)n}\, \| t'_i \| \le\| F ^n_{gx} (t'_i) \| \le \| (F ^n_{gx}  \circ \G_{x} )(t) \| = \| \G_{f^n x} ( F ^n_x (t)) \| \le 
\| \G_{f^n x} \| \cdot e^{(\chi_j-\e)n} \| t \|,
$$
which is impossible since $ \chi_i+\e <\chi_j -\e$. This follows as above from \eqref{epsilon} and 
\eqref{mu} since for a strict sub-resonance $\chi_i < \chi_j$ we have $\chi_i - \chi_j \le \mu<0$ 
and hence $\e<\e_0 \le -\mu/(d+1) \le  -\mu/2<(-\chi_i + \chi_j)/2$.

\vskip.2cm
Now we consider a new family of coordinate changes
$$\tilde \h_x = \G_x^{-1}\circ \h_{gx} \circ \g_x$$ 
which also satisfies $\tilde \h_x(0)=0$ and $D_0\tilde \h_x=\Id$. 
A direct calculation shows that
$$
\begin{aligned}
\tilde \h_{fx} \circ \f_x \circ \tilde \h_x^{-1} & \,=\,  
\G_{fx}^{-1}\circ \h_{fgx} \circ \g_{fx} \circ \f_x \circ  \g_x^{-1}\circ \h_{gx}^{-1} \circ \G_x = \\
&\, = \,\G_{fx}^{-1}\circ \h_{fgx} \circ  \f_{gx} \circ \h_{gx}^{-1} \circ \G_x \,= \,
\G_{fx}^{-1}\circ \p_{gx} \circ \G_x \,= \,
\tilde \p_x.
\end{aligned}
$$
Hence if $ \p_x $ is a sub-resonance polynomial then so is $\tilde \p_x $ as a composition 
of sub-resonance  polynomials. Now part (2) of the theorem gives $\tilde \h_x = G_x \h_x $ for some 
$G_x \in \s_x$ which depends continuously on $x$. Then the definition of $\tilde \h_x$ yields
$$  
\h_{gx} \circ \g_x = \G_x \circ \tilde \h_x = (\G_x G_x) \circ \h_x 
$$
so that $\h_{gx} \circ \g_x \circ \h_x^{-1}=\G_x G_x $, which is again
a sub-resonance polynomial, as claimed.

Similarly, using part (2'), we can obtain that if $ \p_x $,  and hence $\tilde \p_x $, are resonance 
polynomials then  $\h_{gx} \circ \g_x \circ \h_x^{-1}=\G_x G_x $,  where $G_x$, and hence $\G_x  G_x$, 
 are also resonance polynomials for each $x\in X$. 
\vskip.3cm
This completes the proof of Theorem \ref{NFext}.
$\QED$

\subsection {Proof of Corollary \ref{Cinf}}
By part (2) of Theorem \ref{NFext}, if we fix a choice of Taylor polynomials of degree $d$ for $\h_x$, 
then the family $\h_x$ is unique. Then for each $N>d$ we can do the construction in part (1) with 
this fixed choice of Taylor polynomials and obtain the family of $C^{N}$ diffeomorphisms $\h_x$. 
By uniqueness, all these families coincide and hence $\h_x$ are $\Ci$ diffeomorphisms. 
For $\tilde \h$ in part (1') the smoothness follows since, as we will show, it is a composition of $\h$ with a 
polynomial diffeomorphism.

%%%%%%%%%%%%%%%%%%%%%%%%%%%%%%%%%%
%%%%%%%%%%%%%%%%%%%%%%%%%%%%%%%%%%%%
\section{Proof of Theorem \ref {NFfol} and Corollary \ref{cor1/2pinch}}

The parts (1), (1'), (2), (2'), (3), and (5)  of Theorem \ref {NFfol} are obtained obtained 
using Theorem \ref{NFext} as follows. We consider the vector bundle $\E=T\w$ with $\E_x=T_x\w$. 
To construct extension $\f$ as in Theorem \ref{NFext} we restrict $f$ to the leaves of $\w$ and obtain
 $\f_x$ by identifying $B(x,\sigma) \subset T_x \w$ with a neighborhood of $x$ in $\w_x$ using exponential map. It is easy to see that Assumptions \ref{ass} are satisfied with $N=\lfloor r \rfloor$ and 
$\a=r-\lfloor r \rfloor$.

Hence Theorem \ref{NFext} gives existence of families $\{ \h_x\}_{x\in \M}$ and $\{ \h_x'\}_{x\in \M}$
of local normal form coordinates as in (1) and (1') satisfying ``uniqueness properties" (2) and (2').
Then, as in Remark \ref{NFglob}, they can be extended uniquely to global diffeomorphisms $\h_x :  \w_x \to \E_{x}$. 
%This, in particular, provides an identification of a global leaf $\w(x)$ with $\E_x$. 
We note that the H\"older condition at $0$ in Theorem \ref{NFext} implies that $\h_x$ in (1) is
globally H\"older along $\w_x$ by part (4). This also implies that $\h_x'$ in (1') is globally H\"older
since by (2) it differs from $\h_x$ by a polynomial diffeomorphism.

To prove (3), we similarly restrict $g$ to the leaves of $\w$ and obtain the extension $\g$ 
commuting with $\f$, so that the result follows from (3) of Theorem \ref{NFext}.

The existence of $\{ \h_x\}_{x\in \M}$ as in part (5) can be obtained by constructing the Taylor terms
of $\h_x$ that depend smoothly on $x$ as indicated in Remark \ref{smooth dep}, see \cite{KS15} 
for more details of this argument.

Part (4) requires a different argument for which we refer to \cite{KS15,KS16}.
First one shows inductively that the Taylor polynomial of the 
transition maps is sub-resonance, and then argues that error term is zero.

The first part of Corollary \ref{cor1/2pinch} follows directly since the splitting \eqref{splitting} is trivial, 
$d=1$, and hence there are no non-linear sub-resonance polynomials. This also means that 
$\tilde \lambda$ and $\mu$ are not deeded as there are no corresponding relations and we get
$\e_0=\lambda/3=-\chi/3$. Hence we need $\e< \min \{\e_0/3, \nu/( N+\a+1)\}$.
The second term is $-\chi (N+\alpha-1)/(N+\a+1)$ and smallest for $N=1$ and gives $-\chi \a /( 2+\a)$,
which  is less than $\e_0/3$.  This means that we need $\e<  -\chi \a /( 2+\a)$, which yields the interval
$ (-\chi (1+\frac{\a}{2+\a}), -\chi (1-\frac{\a}{2+\a}))$ with endpoint ratio $1+\a$.

%%%%%%%%%%%%%%%%%%%%%%%%%%%%%%%%%%


\begin{thebibliography}{99}

%%%%%%%%%%%%%%%%%%%%%%old


\bibitem[A]{A} L. Arnold.\, Random Dynamical Systems. Springer Monographs in Mathematics. 

\bibitem[AK92]{AK} L. Arnold and X. Kedai.  \, Normal forms for random diffeomorphisms. J. of Dynamics and Differential Equations 4, (1992) Issue 3, 445-483  

%\bibitem[BP]{BP}  L. Barreira and Ya. Pesin.  Nonuniformly Hyperbolicity: Dynamics of Systems with Nonzero Lyapunov Exponents. Encyclopedia of Mathematics and Its Applications, {\bf 115} Cambridge Univ. Press.

\bibitem[BrKo]{BK} I. U. Bronstein and A. Ya. Kopanskii. 
Smooth invariant manifolds and normal forms. World Scientific, 1994.

\bibitem[Bu18]{But} C. Butler. Rigidity of equality of Lyapunov exponents for geodesic flows. J. Differential Geom. 109 (2018), no. 1, 39-79. 

\bibitem[DWX19]{DWX} D. Damjanovic, A. Wilkinson, D. Xu.  Pathology and asymmetry: centralizer rigidity for partially hyperbolic diffeomorphisms. Preprint.

\bibitem[dlLO98]{dlLO}  R. de la Llave and R. Obaya.
{\em Regularity of the composition operator in spaces of H\"older functions.}
Discrete and Continuous Dynamical Systems. 5 (1999), no. 1, 157-184.

\bibitem[F07]{Fa} Y. Fang. On the rigidity of quasiconformal Anosov flows. Ergodic Theory Dynam. Systems 27 (2007), no. 6, 1773-1802.

\bibitem[FFH10]{FFH} Y. Fang, P. Foulon, B. Hasselblatt.\, Zygmund strong foliations in higher dimension. \\J. of Modern Dynamics, 4 (2010), no. 3, 549-569.

\bibitem[Fe95]{F1}  R. Feres. The invariant connection of 1/2-pinched Anosov diffeomorphism and rigidity. Pacific J. Math. Vol. 171 No. 1 (1995), 139-155.

\bibitem[Fe04]{F2}  R. Feres.  A differential-geometric view of normal forms 
of contractions. In Modern Dynamical Systems and Applications, 
%Eds.: M. Brin, B. Hasselblatt, Y. Pesin, 
Cambridge University Press, (2004) 103-121.

%\bibitem[FKSp11]{FKS} D. Fisher, B.  Kalinin,  R.  Spatzier.  Totally nonsymplectic Anosov actions on tori and nilmanifolds. Geometry and Topology, 15 (2011), no. 1, 191-216.

\bibitem[FKSp11]{FKSp10} D. Fisher, B. Kalinin, and R. Spatzier. Totally non-symplectic Anosov actions on tori and nilmanifolds. Geometry and Topology {15} (2011) 191-216.

\bibitem[GKS11]{GKS10}  A. Gogolev, B. Kalinin, and V. Sadovskaya.  Local rigidity for Anosov automorphisms. (with Appendix by R. de la Llave)   Math. Research Letters, 18 (2011), no. 5, 843-858.

\bibitem[GKS19]{GKS19}  A. Gogolev, B. Kalinin, and V. Sadovskaya.  Center foliation rigidity for partially hyperbolic toral diffeomorphisms. Preprint.

\bibitem[Gu02]{G}   M. Guysinsky. The theory of non-stationary normal forms.
              Ergodic Theory  Dynam. Systems, { 22} (3), (2002), 845-862.

\bibitem[GuKt98]{GK}  M. Guysinsky and  A. Katok. Normal forms and invariant 
              geometric structures for dynamical systems with invariant 
              contracting foliations. Math. Research Letters { 5}
              (1998), 149-163.
   
% \bibitem[HPS]{HPS} M. Hirsch, C. Pugh, and M. Shub. Invariant manifolds.                    Springer-Verlag, New York, 1977.

\bibitem[KKt01]{KKt00} B. Kalinin and A. Katok. Invariant measures for actions of higher
rank abelian groups. Proceedings of Symposia in Pure Mathematics. Volume { 69},
(2001), 593-637.

\bibitem[KKt07]{KKt} B. Kalinin and A. Katok.  Measure rigidity beyond uniform 
            hyperbolicity: invariant measures for Cartan actions on tori. 
               J. of Modern Dynamics, { Vol. 1} (2007),  no. 1, 123-146.
   
\bibitem[KKtR11]{KKtR} B. Kalinin, A. Katok, and F. Rodriguez-Hertz. 
         Nonuniform measure rigidity. Annals of Math. 174 (2011), no. 1, 361-400.                

\bibitem[KtL91]{KL}  A. Katok and J. Lewis. Local rigidity for certain groups of 
              toral automorphisms. Israel J. Math. { 75} (1991), 203-241.
              
 \bibitem[KtR15]{KtR} A. Katok and F. Rodriguez-Hertz. 
        Arithmeticity and topology of higher rank actions of Abelian groups.    
         J. of Modern Dynamics, Vol. 10 (2016) 115-152.   
         
\bibitem[KtSp97]{KSp97} A. Katok and R. Spatzier. 
     Differential rigidity of Anosov actions of higher rank abelian groups
    and algebraic lattice actions. 
     Tr. Mat. Inst. Steklova 216 (1997), Din. Sist. i Smezhnye Vopr.,
   292Ð319; translation in Proc. Steklov Inst. Math. 1997, no. 1 (216), 287-314.

\bibitem[KS03]{KS03} B. Kalinin and V. Sadovskaya. On local and global rigidity 
of quasiconformal Anosov diffeomorphisms. J. of the Institute of Mathematics 
of Jussieu (2003) { 2} (4), 567-582.
              
 \bibitem[KS06]{KS} B. Kalinin and V. Sadovskaya. 
 Global rigidity for totally nonsymplectic Anosov $\Z^k$ actions. 
  Geometry and Topology, vol. 10 (2006), 929-954.
 
 %\bibitem[KS07]{KS07} B. Kalinin, V. Sadovskaya.      {\em On classification of resonance-free Anosov $\Z^k$ actions.}           Michigan Math. Journal, {55} (2007), no. 3, 651-670.

\bibitem[KS16]{KS15} B. Kalinin, V. Sadovskaya.   
{\em Normal forms on contracting foliations: smoothness and homogeneous  structure.} 
Geometriae Dedicata, Vol. 183 (2016), no. 1, 181-194.

 \bibitem[KS17]{KS16} B. Kalinin, V. Sadovskaya.   {\em Normal forms for non-uniform contractions}.  
Journal of Modern Dynamics, vol. 11 (2017), 341-368.  

%\bibitem[LY85]{LY} F. Ledrappier and L.-S. Young, {\it  The metric entropy of diffeomorphisms. I. Characterization of measures satisfying Pesin's  entropy formula}, Annals of Math. {122}, (1985) no. 3, 509--539.

\bibitem[LL05]{LL} W. Li and  K. Lu. Sternberg theorems for random dynamical systems. Communications on Pure and Applied Mathematics, 
Vol. LVIII  (2005), 0941-0988.

\bibitem[M19]{M} K. Melnick.
{\em Nonstationary smooth geometric structures for contracting measurable cocycles}
Ergodic Theory and Dynamical Systems, 39 no. 2 (2019) 392-424.

 \bibitem[P]{P} Ya. Pesin. Lectures on Partial Hyperbolicity and Stable Ergodicity. 
Zurich Lectures in Advanced Mathematics, EMS, 2004.
  
\bibitem[R79]{R} D. Ruelle. Ergodic theory of differentiable dynamical systems. Publications Math\'ematiques de l'I.H.\' E.S. 50 (1979), 27-58.

\bibitem[S05]{S} V. Sadovskaya. On uniformly quasiconformal Anosov systems. 
Math. Research Letters, vol. 12 (2005), no. 3, 425-441.

\bibitem[St57]{St} S. Sternberg. Local contractions and a theorem of Poincar\'e. 
            Amer. J. of Math. 79 (1957), 809-824. 

\end{thebibliography}
\end{document}